\documentclass[12pt]{amsart}

\usepackage[hmargin=0.8in,height=8.6in]{geometry}
\usepackage{amssymb,amsthm, times}
\usepackage{delarray,verbatim}
\usepackage{ifpdf}
\ifpdf
\usepackage[pdftex]{graphicx}
 \else
\usepackage[dvips]{graphicx}
 \fi

\usepackage{bm}

\usepackage{ifpdf}
\usepackage{color}
\definecolor{webgreen}{rgb}{0,.5,0}
\definecolor{webbrown}{rgb}{.8,0,0}
\definecolor{emphcolor}{rgb}{0.95,0.95,0.95}

\usepackage{hyperref}
\hypersetup{%
          colorlinks=true,
          linkcolor=webbrown,
          filecolor=webbrown,
          citecolor=webgreen,
          breaklinks=true}
\ifpdf \hypersetup{pdftex,
            pdfstartview=FitH, 
            bookmarksopen=true,
            bookmarksnumbered=true
} \else \hypersetup{dvips} \fi

\newcommand{\lapinv}{\Phi(q)}
\newcommand {\B}{\mathcal{B}}

\numberwithin{equation}{section}

\newtheorem{theorem}{Theorem}[section]
\newtheorem{proposition}{Proposition}[section]

\newtheorem{remark}{Remark}[section]
\newtheorem{lemma}{Lemma}[section]

\newtheorem{assump}{Assumption}[section]
\newtheorem{definition}{Definition}[section]
\numberwithin{remark}{section} \numberwithin{proposition}{section}
\numberwithin{corollary}{section}
\newcommand {\R}{\mathbb{R}}

\newcommand {\p}{\mathbb{P}}

\newcommand {\E}{\mathbb{E}}

\newcommand{\diff}{{\rm d}}

\newcommand{\lev}{L\'{e}vy }

\title[Doubly reflected L\'EVY processes in singular control]{Optimality of doubly reflected L\'EVY processes in singular control}
\thanks{This version: \today. }
\thanks{$*$\, Department of Statistics, London School of Economics, Houghton Street, London, WC2A 2AE, UK. Email: E.J.Baurdoux@lse.ac.uk.}
\thanks{$\dagger$\,  (corresponding author) Department of Mathematics,
Faculty of Engineering Science, Kansai University, 3-3-35 Yamate-cho, Suita-shi, Osaka 564-8680, Japan. Email: \mbox{{\em
kyamazak@kansai-u.ac.jp}}.  Tel: +81-6-6368-1527. }
\author[E. J. Baurdoux]{Erik J. Baurdoux$^*$}
\author[K. Yamazaki]{Kazutoshi Yamazaki$^\dagger$}
\date{}
\begin{document}

\begin{abstract}
We consider a class of two-sided singular control problems.  A controller either increases or decreases a given spectrally negative \lev process so as to minimize the total costs comprising of the running and controlling costs where the latter is proportional to the size of control.  We provide a sufficient condition for the optimality of a double barrier strategy, and in particular show that it holds when the running cost function is convex. Using the fluctuation theory of doubly reflected \lev processes, we express concisely the optimal strategy as well as the value function using the scale function. Numerical examples are provided to confirm the analytical results.
\\
\noindent \small{\noindent  AMS 2010 Subject Classifications: 60G51, 93E20, 49J40 
}\\
\end{abstract}

\maketitle

\section{Introduction}
We consider the problem of optimally modifying a stochastic process by means of singular control.  An admissible strategy is two-sided and the process can be increased or decreased.  The objective is to minimize the expected total costs comprising of the running and controlling costs; the former is modeled as some given function $f$ of the controlled process that is accumulated over time, and the latter is proportional to the size of control.  The problem of singular control arises in various contexts.  For its applications, we refer the reader to, e.g.,
 \cite{dai2013brownian1, dai2013brownian} for  inventory management, \cite{eppen1969cash} for cash balance management, \cite{bayraktar2008analysis, karatzas1981monotone} for monotone follower problems and \cite{Avram_et_al_2007,Bayraktar_2012,guo2005optimal, Loeffen_2008, merhi2007model, sethi2002optimal} for finance and insurance.

%

This paper studies a spectrally negative \lev model where the underlying process, in the absence of control, follows a general \lev process with only negative jumps.  We pursue a sufficient condition on the running cost function $f$ such that a strategy of double barrier type is optimal and the value function is obtained semi-explicitly.    This generalizes the classical Brownian motion model \cite{MR716123} and complements the results on the continuous diffusion model as in \cite{matomaki2012solvability}.

%
%
%



Motivated by the recent research of spectrally negative \lev processes and their applications, we take advantage of their fluctuation theory as in \cite{Bertoin_1996, Kyprianou_2006}. These techniques are used extensively in stochastic control problems in the last decade.  Exemplifying examples include de Finetti's dividend problem as in \cite{Avram_et_al_2007,Bayraktar_2012,Loeffen_2008}, where a single barrier strategy is shown to be optimal under certain conditions.     In these papers, the so-called \emph{scale function} is commonly used to express the net present value of the barrier strategy.  Thanks to its analytical properties such as continuity/smoothness (see, e.g., \cite{Hubalek_Kyprianou_2009,Kyprianou_2006}), the selection of the candidate barrier level and the verification of optimality can be carried out efficiently. While a part of the verification is still problem-dependent and is often a difficult task, 
these methods allow one to solve for this wide class of \lev processes without specializing on a particular type, whether or not the process is of infinite activity/variation.

This paper considers a variant of the above mentioned papers where the control is allowed to be two-sided.  Our objective is to show the optimality of a double barrier strategy where the resulting controlled process becomes a \emph{doubly reflected \lev process} of \cite{Avram_et_al_2007, Pistorius_2004}.  Existing research on the optimality of doubly reflected \lev processes includes the dividend problem with capital injection as in \cite{Avram_et_al_2007, Bayraktar_2012}.  Other related problems where two threshold levels characterize the optimal strategy include stochastic games  \cite{BauKyrianouPardo2011, Baurdoux2008, Leung_Yamazaki_2011, Hernandez_Yamazaki_2013} and impulse control \cite{Bayraktar_2013, Yamazaki_2013}.  


In this paper, we take the following steps to achieve our goal:
\begin{enumerate}
 \item We first write via the scale function the expected net present value corresponding to the double barrier strategy; this is a direct application of the results in \cite{Avram_et_al_2007,Pistorius_2004}.
\item This is followed by the selection of the two barriers.  The upper barrier is chosen so that the resulting candidate value function becomes twice differentiable at the barrier; the lower barrier is chosen so that it is continously (resp.\ twice) differentiable when the process is of bounded (resp.\ unbounded variation). 
\item We then analyze the existence of such a pair that satisfy the two conditions simultaneously.  We show that either such a pair exist, or otherwise a single barrier strategy (with the upper barrier set to infinity) is optimal.
\item In order to verify the optimality of the strategy defined in the previous steps, we study the verification lemma and identify some additional conditions that are sufficient for the optimality.  Moreover, we show that it is satisfied whenever the running cost function $f$ is convex.
\end{enumerate}
As in the above mentioned papers, we use the special known properties of the scale function to solve the problem. In particular, the steps taken here are similar to those used in \cite{Hernandez_Yamazaki_2013}, where two parameters are shown to characterize the optimal strategies in the two-person game they considered.  The main novelty and challenge here are that we solve the problem without specifying the form of the running cost function $f$ and derive a most general condition on $f$ that is sufficient for the optimality of a doubly reflected \lev process.

In addition to the above, we give examples with (piecewise) quadratic and linear cases for $f$, which have been used in, e.g., \cite{baccarin2002optimal, constantinides1978existence, sulem1986solvable}.   We shall see in particular that in the linear case the upper boundary can become infinity (or equivalently a single barrier strategy is optimal), whereas it does not occur in the quadratic case.   
In order to confirm the obtained analytical results, we give numerical examples where the underlying process is a spectrally negative \lev process in the $\beta$-family of Kuznetsov \cite{Kuznetsov_2010_2}.  

The rest of the paper is organized as follows.  Section \ref{section_model} gives a mathematical model of the problem and  a brief review on the spectrally negative \lev process and the scale function.  Section \ref{section_double_barier_strategy} expresses via the scale function the expected net present value under the double barrier strategy.  The candidate barrier levels are then selected by using the smoothness conditions at the barriers, and the existence of such a pair is shown.
 In Section  \ref{section_verification}, we study the verification lemma for this problem and analyze what additional conditions are required for the candidate value function to be optimal. 
 Section \ref{section_sufficient_condition}   obtains a more concrete sufficient condition and in particular shows that it is satisfied when $f$ is convex.  In Section \ref{section_examples}, we give examples with piecewise quadratic and linear cases. We conclude the paper with numerical examples in Section \ref{section_numerics}.  

Throughout the paper,  $x+$ and $x-$ are used to indicate the right and left hand limits, respectively.  The superscripts $x^+ := \max(x, 0)$, $f^+(x) := \max(f(x), 0)$, $x^- := \max(-x, 0)$ and $f^-(x) := \max(-f(x), 0)$ are used to indicate positive and negative parts.  Monotonicity is understood in the strict sense; for the weak sense  ``nondecreasing" and ``nonincreasing" are used. The convexity (unless otherwise stated) is in the weak sense.

\section{Mathematical Formulation} \label{section_model}

Let $(\Omega, \mathcal{F}, \p)$ be a probability space hosting a \emph{spectrally negative \lev process} $X = \left\{X_t; t \geq 0 \right\}$ whose \emph{Laplace exponent} is given by
\begin{align}
\psi(s)  := \log \E \left[ e^{s X_1} \right] =  c s +\frac{1}{2}\sigma^2 s^2 + \int_{(-\infty,0)} (e^{s z}-1 - s z 1_{\{-1 < z < 0\}}) \nu (\diff z), \quad s \geq 0, \label{laplace_spectrally_positive}
\end{align}
where $\nu$ is a \lev measure with the support $(-\infty,0)$ that satisfies the integrability condition $\int_{(-\infty,0)} (1 \wedge z^2) \nu(\diff z) < \infty$.  It has paths of bounded variation if and only if $\sigma = 0$ and $\int_{(-1,0)}|z|\, \nu(\diff z) < \infty$; in this case, we write \eqref{laplace_spectrally_positive} as
\begin{align*}
\psi(s)   =  \delta s +\int_{(-\infty,0)} (e^{s z}-1 ) \nu (\diff z), \quad s \geq 0,
\end{align*}
with $\delta := c - \int_{(-1,0)}z\, \nu(\diff z)$.  We exclude the case in which $X$ is the negative of a subordinator (i.e., $X$ has monotone paths a.s.). This assumption implies that $\delta > 0$ when $X$ is of bounded variation. 
Let $\mathbb{P}_x$ be the conditional probability under which $X_0 = x$ (also let $\mathbb{P} \equiv \mathbb{P}_0$), and let $\mathbb{F} := \left\{ \mathcal{F}_t, t \geq 0 \right\}$ be the filtration generated by $X$.  

An admissible strategy $\pi := \left\{ (U_t^{\pi}, D_t^{\pi}); t \geq 0 \right\}$ is given by a pair of nondecreasing, right-continuous and $\mathbb{F}$-adapted processes such that $U^\pi_{0-}=D^\pi_{0-}=0$ and, as is assumed in \cite{MR716123},
\begin{align}
\E_x \Big[ \int_{[0,\infty)} e^{-qt} (\diff U^\pi_{t} + \diff D^\pi_{t})  \Big] < \infty, \quad x \in \R. \label{feasibility_L_R}
\end{align}
Let $\Pi$ be the set of all admissible strategies and the discount $q$ is assumed to be a strictly positive constant.

 With the controlled process $Y_{t}^\pi := X_t + U_t^\pi - D_t^\pi$, $t \geq 0$, the problem is to compute the total expected costs:
\begin{align*}
v^\pi (x) := \E_x \Big[ \int_0^\infty e^{-qt}  f (Y_{t}^\pi) \diff t  + \int_{[0,\infty)}e^{-qt}\left(C_U \diff U^\pi_{t} + C_D \diff D^\pi_{t} \right) \Big], \quad x \in \R,
\end{align*}
for some running cost function $f$ satisfying the conditions specified below and fixed constants $C_U, C_D \in \R$ such that 
\begin{align}
C_U + C_D > 0, \label{assump_C_sum}
\end{align}
and to obtain an admissible strategy over $\Pi$ that minimizes it, if such a strategy exists.  The inequality \eqref{assump_C_sum} is commonly assumed in the literature (see, e.g., \cite{MR716123, matomaki2012solvability}); this implies that it is suboptimal to activate  $U_t^\pi$ and $D_t^\pi$ simultaneously. Hence, we can safely assume that the supports of the Stieltjes measures $\diff U_t^\pi (\omega)$ and $\diff D_t^\pi (\omega)$ do not overlap for a.e.\ $\omega \in \Omega$.

Regarding the running cost function $f$, we assume the same assumptions as in \cite{Bensoussan_2009,Bensoussan_2005, Yamazaki_2013}; this is a crucial condition when dealing with a process with negative jumps.
\begin{assump}  \label{assump_f_g}  We assume that $f : \R \rightarrow \R$ satisfies the following.
\begin{enumerate}
\item $f$ is continuous and is a piecewise continuously differentiable function and grows (or decreases) at most polynomially (in the sense defined by Beyer et al. \cite{Beyer_1998}).
\item There exists a number $\overline{a} \in \R$ such that the function
\begin{align}
\tilde{f}(x) &:= f(x) + C_U  q x, \quad x \in \R, \label{def_f_tilde}
\end{align}
is increasing on $(\overline{a},\infty)$ and is decreasing and convex on $(-\infty, \overline{a})$. 
\item There exist a $c_0 > 0$ and an $x_0 \geq \overline{a}$ such that $\tilde{f}'(x) \geq c_0$ for $x \geq x_0$.
\end{enumerate}
\end{assump}

For the problem to make sense, we assume that the \lev process $X$ has a finite moment.
\begin{assump}  \label{assump_finiteness_mu}We assume that $\E [X_1]  = \psi'(0+) \in (-\infty,\infty)$. 
\end{assump}

\subsection{Scale functions}

Fix $q > 0$. For any spectrally negative \lev process, there exists a function called  the  $q$-scale function 
\begin{align*}
W^{(q)}: \R \rightarrow [0,\infty), 
\end{align*}
which is zero on $(-\infty,0)$, continuous and increasing on $[0,\infty)$, and is characterized by the Laplace transform:
\begin{align*}
\int_0^\infty e^{-s x} W^{(q)}(x) \diff x = \frac 1
{\psi(s)-q}, \qquad s > \lapinv,
\end{align*}
where
\begin{equation}
\lapinv :=\sup\{\lambda \geq 0: \psi(\lambda)=q\}. \notag
\end{equation}
Here, the Laplace exponent $\psi$ in \eqref{laplace_spectrally_positive} is known to be zero at the origin and strictly convex on $[0,\infty)$; therefore $\lapinv$ is well defined and is strictly positive as $q > 0$.   We also define, for $x \in \R$,
\begin{align*}
\overline{W}^{(q)}(x) &:=  \int_0^x W^{(q)}(y) \diff y, \\
Z^{(q)}(x) &:= 1 + q \overline{W}^{(q)}(x),  \\
\overline{Z}^{(q)}(x) &:= \int_0^x Z^{(q)} (z) \diff z = x + q \int_0^x \int_0^z W^{(q)} (w) \diff w \diff z.
\end{align*}
Because $W^{(q)}$ is uniformly zero on the negative half line, we have
\begin{align}
Z^{(q)}(x) = 1  \quad \textrm{and} \quad \overline{Z}^{(q)}(x) = x, \quad x \leq 0.  \label{z_below_zero}
\end{align}

Let us define the first down- and up-crossing times, respectively, of $X$ by
\begin{align*}
\tau_b^- := \inf \left\{ t \geq 0: X_t < b \right\} \quad \textrm{and} \quad \tau_b^+ := \inf \left\{ t \geq 0: X_t >  b \right\}, \quad b \in \R.
\end{align*}
Then, for any $b > 0$ and $x \leq b$,
\begin{align}
\label{two_sided_exit}
\E_x \left[ e^{-q \tau_b^+} 1_{\left\{ \tau_b^+ < \tau_0^- \right\}}\right] = \frac {W^{(q)}(x)}  {W^{(q)}(b)} \quad \textrm{and} \quad
 \E_x \left[ e^{-q \tau_0^-} 1_{\left\{ \tau_b^+ > \tau_0^- \right\}}\right] = Z^{(q)}(x) -  Z^{(q)}(b) \frac {W^{(q)}(x)}  {W^{(q)}(b)}.
\end{align}
By taking limits on the latter,
\begin{align*} 
 \E_x \left[ e^{-q \tau_0^-} \right] &= Z^{(q)}(x) -  \frac q {\Phi(q)} W^{(q)}(x), \quad x \in \R.
\end{align*}

Fix $\lambda \geq 0$ and define $\psi_\lambda(\cdot)$ as the Laplace exponent of $X$ under $\p^\lambda$ with the change of measure 
\begin{align*}
\left. \frac {\diff \p^\lambda} {\diff \p}\right|_{\mathcal{F}_t} = \exp(\lambda X_t - \psi(\lambda) t), \quad t \geq 0; 
\end{align*}
see page 213 of \cite{Kyprianou_2006}.
 Suppose $W_\lambda^{(q)}$ and $Z_\lambda^{(q)}$ are the scale functions associated with $X$ under $\p^\lambda$ (or equivalently with $\psi_\lambda(\cdot)$).  
Then, by Lemma 8.4 of \cite{Kyprianou_2006}, $W_\lambda^{(q-\psi(\lambda))}(x) = e^{-\lambda x} W^{(q)}(x)$, $x \in \R$,
which is well defined even for $q \leq \psi(\lambda)$ by Lemmas 8.3 and 8.5 of \cite{Kyprianou_2006}.  In particular, we define
\begin{align}
W_{\Phi(q)}(x) := W_{\Phi(q)}^{(0)}(x) = e^{-\Phi(q) x} W^{(q)}(x), \quad x \in \R, \label{scale_function_version}
\end{align}
which is known to be an increasing function and,  as in Lemma 3.3 of \cite{Kuznetsov2013}, 
\begin{align}
W_{\Phi(q)} (x) \nearrow \psi'(\Phi(q))^{-1} \quad \textrm{as } x \rightarrow \infty. \label{W_q_limit}
\end{align}

\begin{remark} \label{remark_smoothness_zero}
\begin{enumerate}
\item If $X$ is of unbounded variation or the \lev measure is atomless, it is known that $W^{(q)}$ is $C^1(\R \backslash \{0\})$; see, e.g.,\ \cite{Chan_2009}.  Hence, 
\begin{enumerate}
\item $Z^{(q)}$ is $C^1 (\R \backslash \{0\})$ and $C^0 (\R)$ for the bounded variation case, while it is $C^2(\R \backslash \{0\})$ and $C^1 (\R)$ for the unbounded variation case, and
\item $\overline{Z}^{(q)}$ is $C^2(\R \backslash \{0\})$ and $C^1 (\R)$ for the bounded variation case, while it is $C^3(\R \backslash \{0\})$ and $C^2 (\R)$ for the unbounded variation case.
\end{enumerate}
\item Regarding the asymptotic behavior near zero, as in Lemmas 4.3 and 4.4 of \cite{Kyprianou_Surya_2007}, 
\begin{align}\label{eq:Wqp0}
\begin{split}
W^{(q)} (0) &= \left\{ \begin{array}{ll} 0, & \textrm{if $X$ is of unbounded
variation,} \\ \frac 1 {\delta}, & \textrm{if $X$ is of bounded variation,}
\end{array} \right. \\
W^{(q)'} (0+) &:= \lim_{x \downarrow 0}W^{(q)'} (x) =
\left\{ \begin{array}{ll}  \frac 2 {\sigma^2}, & \textrm{if }\sigma > 0, \\
\infty, & \textrm{if }\sigma = 0 \; \textrm{and} \; \nu(-\infty,0) = \infty, \\
\frac {q + \nu(-\infty, 0)} {\delta^2}, &  \textrm{if }\sigma = 0 \; \textrm{and} \; \nu(-\infty, 0) < \infty.
\end{array} \right.
\end{split}
\end{align}
\item As in (8.18) and Lemma 8.2 of \cite{Kyprianou_2006},
\begin{align*}
\frac {W^{(q)'}(y+)} {W^{(q)}(y)} \leq \frac {W^{(q)'}(x+)} {W^{(q)}(x)},  \quad  y > x > 0.
\end{align*}
In all cases, $W^{(q)'}(x-) \geq W^{(q)'}(x+)$ for all $x >0$.
\end{enumerate}
\end{remark}

The problem in this paper is a  generalization of Section 6 of \cite{Yamazaki_2013}, where $D_t^\pi$ is restricted to be zero.  Its optimal solution  is a (single) barrier strategy, which is described immediately below.  Define, for any measurable function $h$ and $s \in \R$,
\begin{align*} 
\begin{split}
\Psi(s;h) &:= \int_0^\infty e^{- \Phi(q) y}   h(y+s) \diff y = \int_s^\infty e^{- \Phi(q) (y-s)}   h(y) \diff y, \\
\varphi_s (x;h) &:= \int_{s}^x W^{(q)} (x-y) h(y) \diff y, \quad x \in \R.
\end{split}
\end{align*}
Here $\varphi_s (x;h) = 0$ for any $x \leq s$ because $W^{(q)}$ is uniformly zero on $(-\infty,0)$.

 The following, which holds directly from Assumption \ref{assump_f_g},  is due to \cite{Bensoussan_2005}.  Here note that $\Psi(\cdot;\tilde{f}')$ is equivalent to (4.23) of \cite{Bensoussan_2005} (times a positive constant).  While in \cite{Bensoussan_2005}, they focus on a special class of spectrally negative \lev processes, the results still hold for a general spectrally negative \lev process.
\begin{lemma}[Proposition 5.1 of \cite{Bensoussan_2005}] \label{lemma_bensoussan_result}
\begin{enumerate}
\item There exists a unique number $\underline{a} < \overline{a}$ such that $\Psi(\underline{a}; \tilde{f}') = 0$,  $\Psi(x;\tilde{f}') < 0$ if $x < \underline{a}$ and $\Psi(x; \tilde{f}') > 0$ if $x > \underline{a}$.
\item $\Psi'(x; \tilde{f}') > 0$ for $x \leq \overline{a}$. 
\item $\Psi(x; \tilde{f}')  \geq c_0/\Phi(q)$ for $x \geq x_0$. 
\end{enumerate}
\end{lemma}
Namely, while $\overline{a}$ is the unique zero of $\tilde{f}'$, $\underline{a}$ is the unique zero of $\Psi(\cdot; \tilde{f}')$. We are now ready to state the results of the auxiliary problem.

\begin{theorem}[Theorem 6.1 of \cite{Yamazaki_2013}] \label{theorem_main_no_transaction}
Consider a version of the problem that minimizes
\begin{align*}
\tilde{v}^\pi (x) := \E_x \Big[ \int_0^\infty e^{-qt}  f (\widetilde{Y}_{t}^\pi) \diff t  + \int_{[0,\infty)} e^{-qt} C_U\diff U^\pi_{t}  \Big], \quad x \in \R,
\end{align*}
with the controlled process $\widetilde{Y}_t^\pi := X_t + U_t^\pi$, $t \geq 0$ for some fixed constant $C_U \in \R$, that satisfies Assumption \ref{assump_f_g}.

Then the barrier strategy $U^{\underline{a}, \infty}$ defined by
\begin{align}
U^{\underline{a}, \infty}_t := \sup_{0 \leq t' \leq t} ({\underline{a}}-X_{t'}) \vee 0, \quad t \geq 0, \label{single_reflected}
\end{align}
is optimal and the value function is 
\begin{align} \label{value_function_no_cost_case}
\begin{split}
\inf_{\pi} \tilde{v}^\pi (x) = \tilde{v}_{\underline{a}}(x) 
&:=  -  C_U \left( \overline{Z}^{(q)}(x-\underline{a}) + \frac {\psi'(0+)} q  \right) + 
\frac {f(\underline{a})} q Z^{(q)}(x-\underline{a}) - \varphi_{\underline{a}} (x;f).
\end{split}
\end{align}

\end{theorem}

\section{The double barrier strategies} \label{section_double_barier_strategy}

Following Avram et al.\ \cite{Avram_et_al_2007} and Pistorius \cite{Pistorius_2003}, we define a doubly reflected \lev process given by 
\begin{align*}
Y_t^{a,b} := X_t + U_t^{a,b} - D_t^{a,b}, \quad t \geq 0, \; a < b,
\end{align*}
which is reflected at two barriers $a$ and $b$ so as to stay on the interval $[a,b]$;  see page 165 of \cite{Avram_et_al_2007} for the construction of this process.  We let $\pi_{a,b}$ be the corresponding strategy and $v_{a,b}$ the corresponding expected total cost. Our aim is to show that by choosing the values of $(a,b)$ appropriately, the minimization is attained by the strategy $\pi_{a,b}$.

For $b \geq a$, let 
\begin{align} \label{about_gamma}
\begin{split}
\Gamma(a,b) &:= C_D + C_U Z^{(q)}(b-a) +  f(b) W^{(q)}(0) + \int_a^b  f(y)  {W^{(q)'}(b-y)} \diff y - W^{(q)}(b-a) f(a) \\
&= C_D + C_U + \varphi_a (b;\tilde{f}').
\end{split}
\end{align}
Also let
\begin{align*}
R^{(q)}(y)  := \overline{Z}^{(q)}(y) + \frac {\psi'(0+)} q, \quad y \in \R.
\end{align*}

\begin{lemma} Fix any $a < b$.  We have $\pi_{a,b} \in \Pi$.  Moreover, for $x \leq b$,
\begin{align} \label{expression_v_a_b}
\begin{split}
v_{a,b} (x) 
&=\frac  {\Gamma(a,b)} {q W^{(q)}(b-a)}  {Z^{(q)}(x-a)} -C_U R^{(q)}(x-a) + \frac {f(a)} q Z^{(q)} (x-a) - \varphi_{a}(x; f) \\
&= \frac {\Gamma(a,b) } {q W^{(q)}(b-a)}{Z^{(q)}(x-a)} -C_U R^{(q)}(x-a) + \frac {f(a)} q - \int_a^x \overline{W}^{(q)}(x-y) f'(y) \diff y.
\end{split}
\end{align}

For $x \geq b$, we have $v_{a,b} (x) = v_{a,b} (b) + C_D (x-b)$.
\end{lemma}
\begin{proof}
As in Theorem 1 of \cite{Avram_et_al_2007}, for all $a \leq x \leq b$,
\begin{align*}
\E_x \left[ \int_{[0,\infty)} e^{-qt} \diff D_t^{a,b} \right] &= \frac {Z^{(q)}(x-a)} {q W^{(q)}(b-a)}, \\
\E_x \left[ \int_{[0,\infty)} e^{-qt} \diff U_t^{a,b} \right] &= - R^{(q)}(x-a)  + \frac {Z^{(q)}(b-a)} {q W^{(q)}(b-a)} Z^{(q)}(x-a),
\end{align*}
which are finite under Assumption \ref{assump_finiteness_mu} and hence $\pi_{a,b} \in \Pi$.
The $q$-resolvent density of $Y^{a,b}$ is, by Theorem 1 of \cite{Pistorius_2003}, for $y \in [a,b]$,\begin{align*}
\E_x \left[ \int_{[0,\infty)} e^{-qt} 1_{\{Y_t^{a, b} \in \diff y \}} \diff t \right] &= \left[ \frac {Z^{(q)}(x-a) W^{(q)'}(b-y)} {q W^{(q)}(b-a)} - W^{(q)}(x-y) \right] \diff y \\ &+  \Big[ Z^{(q)}(x-a) \frac {W^{(q)}(0)} {q W^{(q)}(b-a)} \Big] \delta_b(\diff y),
\end{align*}
where $\delta_b$ is the Dirac measure at $b$. Summing up these,
\begin{align*}
v_{a,b} (x) &= C_D\frac {Z^{(q)}(x-a)} {q W^{(q)}(b-a)} + C_U  \Big[ - R^{(q)}(x-a)  + \frac {Z^{(q)}(b-a)} {q W^{(q)}(b-a)} Z^{(q)}(x-a) \Big]  \\
&+ \int_a^b  f(y) \Big[ \frac {Z^{(q)}(x-a) W^{(q)'}(b-y)} {q W^{(q)}(b-a)} - W^{(q)}(x-y) \Big] \diff y + f(b) Z^{(q)}(x-a) \frac {W^{(q)}(0)} {q W^{(q)}(b-a)} \\
&= \frac {Z^{(q)}(x-a)} {q W^{(q)}(b-a)} \Big[ C_D + C_U Z^{(q)}(b-a) +  f(b) W^{(q)}(0) + \int_a^b  f(y)  {W^{(q)'}(b-y)} \diff y\Big] \\
&-C_U R^{(q)}(x-a) -\varphi_a (x;f),
\end{align*}
which equals the first equality of \eqref{expression_v_a_b}.  The second equality holds because 
integration by parts gives
\begin{align*}
\varphi_a (x;f) = \overline{W}^{(q)}(x-a) f(a) + \int_a^x \overline{W}^{(q)}(x-y) f'(y) \diff y, \quad x \geq a. 
\end{align*}

The case $x < a$ holds because $v_{a,b} (x) = v_{a, b}(a)  - C_U (x-a)$ and $Z^{(q)}(x-a) = Z^{(q)}(0) $ and $R^{(q)}(x-a) = (x-a) + R^{(q)}(0)$.  The case $x > b$ similarly holds.
\end{proof}

\subsection{Smoothness conditions}
Taking a derivative in \eqref{expression_v_a_b}, 
\begin{align} \label{v_derivative_general}
v_{a,b}' (x)&= \frac { \Gamma(a,b)} {W^{(q)}(b-a)} {W^{(q)}(x-a)} -C_U  -\varphi_a (x;\tilde{f}'), \quad a < x < b,
\end{align}
and hence by \eqref{about_gamma}
\begin{align}\label{smooth_fit_all_b}
\begin{split}
v_{a,b}' (b-)= C_D \quad \textrm{and} \quad
v_{a,b}' (a+)=   \frac {\Gamma(a,b)} {W^{(q)}(b-a)}  {W^{(q)}(0)} -C_U. 
\end{split}
\end{align}
This implies, in view of  Remark \ref{remark_smoothness_zero}(2), that the differentiability of $v_{a,b}$ at $b$ holds  for all cases while it holds at $a$ when $\Gamma(a,b)/W^{(q)}(b-a) = 0$ for the case of bounded variation and it holds automatically for the case of unbounded variation.

Taking another derivative, we have, for a.e.\ $x \in (a,b)$,
\begin{align*}
v_{a,b}'' (x)&= \frac {W^{(q)'}(x-a)} {W^{(q)}(b-a)} \Gamma(a,b)  -\int_a^x W^{(q)'}(x-y) \tilde{f}'(y) \diff y -  \tilde{f}'(x) W^{(q)}(0),
\end{align*}
and hence
\begin{align*}
v_{a,b}'' (b-) &= \frac {\Gamma(a,b) } {W^{(q)}(b-a)} {W^{(q)'}((b-a)-)}  -\gamma(a,b), \\
v_{a,b}'' (a+)&= \frac {\Gamma(a,b) }  {W^{(q)}(b-a)}  {W^{(q)'}(0+)} - \tilde{f}'(a+) W^{(q)}(0) ,
\end{align*}
where
\begin{align}
\gamma(a,b) :=  \int_a^b W^{(q)'}(b-y) \tilde{f}'(y) \diff y +  \tilde{f}'(b-) W^{(q)}(0)= \frac \partial {\partial b} \Gamma(a,b-), \quad b > a. \label{small_gamma}
\end{align}
Hence, our candidate levels $(a,b)$ are such that
\begin{align}
\frac {\Gamma(a,b)} {W^{(q)}(b-a)} &= 0,  \label{smoothness_condition1} \\
\gamma(a,b) &= 0. \label{smoothness_condition2}
\end{align}
Here we understand for the case $b =\infty$ that $\lim_{b \rightarrow \infty} \Gamma(a,b)/W^{(q)}(a,b) = 0$. In such case, in view of \eqref{expression_v_a_b}, $\lim_{b \rightarrow \infty }v_{a,b} (x) = v_{a, \infty} (x) := -C_U R^{(q)}(x-a) +  {f(a)}  Z^{(q)} (x-a) /q - \varphi_{a}(x; f)$ and hence $\lim_{b \rightarrow \infty }v_{a,b}'(x) = v_{a, \infty}'(x)$ and $\lim_{b \rightarrow \infty }v_{a,b}'' (x) = v_{a, \infty}'' (x)$ also hold.

We summarize the results as follows.
\begin{lemma} \label{lemma_smoothfit}
\begin{enumerate}
\item If \eqref{smoothness_condition1} holds for some $a < b \leq\infty$,
 then $v_{a,b}$ is differentiable (resp.\ twice-differentiable) at $a$ when $X$ is of bounded (resp.\ unbounded) variation. 
\item If in addition $b < \infty$, it is continuously differentiable at $b$.  In particular, when \eqref{smoothness_condition2} further holds, then it is twice-differentiable at $b$. \end{enumerate}
\end{lemma}

\subsection{Existence of $(a^*, b^*)$}  \label{subsection_existence_a_b}
Here we show the existence of a pair $(a^*, b^*)$ where \eqref{smoothness_condition1} and \eqref{smoothness_condition2} hold simultaneously. Equivalently, we pursue $(a^*, b^*)$ such that the function $b \mapsto \Gamma(a^*, b)$ attains a minimum $0$ at $b^*$ (if $b^* < \infty$).
%

First, by \eqref{assump_C_sum}, \eqref{about_gamma} and \eqref{small_gamma},
\begin{align}
\Gamma(a,a) =  C_D+C_U > 0  \quad \textrm{and} \quad
\gamma(a,a+) =  \tilde{f}'(a+) W^{(q)}(0), \quad a \in \R. \label{Gamma_initial} 
\end{align}
Recall the definition of the level $\overline{a}$ as in Assumption \ref{assump_f_g}. Fix any  $a \geq \overline{a}$. Because $\gamma(a,b) \geq 0$ for $b > a$ in view of \eqref{small_gamma}, the function $b \mapsto \Gamma(a,b)$ starts at a positive value $\Gamma(a,a)$ and increases in $b$.  Therefore, it never crosses nor touches the x-axis.

We now start at $\overline{a}$ and decrease its value until we arrive at the desired pair $(a^*,b^*)$ such that $b \mapsto \Gamma(a^*, b)$ first touches  the x-axis at $b^*$.   

We shall first show that $\underline{a}$ as in Lemma \ref{lemma_bensoussan_result}(1) becomes a lower bound of such  $a^*$.
For any fixed $a \in \R$,
\begin{align*}
\frac {\varphi_a (b;\tilde{f}')} {W^{(q)}(b-a)} =  \int_a^b \tilde{f}'(y)  \frac { {W^{(q)}(b-y)} } {W^{(q)}(b-a)} \diff y = e^{\Phi(q)a} \int_a^b \tilde{f}'(y)  e^{-\Phi(q)y} \frac { {W_{\Phi(q)}(b-y)} } {W_{\Phi(q)}(b-a)} \diff y. 
\end{align*}
Here, by \eqref{W_q_limit}, $1_{\{ y \leq b \}} |\tilde{f}'(y)|  e^{-\Phi(q)y}  { {W_{\Phi(q)}(b-y)} } / {W_{\Phi(q)}(b-a)} \leq  |\tilde{f}'(y)|  e^{-\Phi(q)y}$ for a.e.\ $y \geq a$, which is integrable over $(a, \infty)$ by Assumption \ref{assump_f_g}(1).  Hence, by \eqref{about_gamma}, dominated convergence gives
\begin{align}
\lim_{b \rightarrow \infty} \frac {\Gamma(a,b)} {W^{(q)}(b-a)} = \Psi(a; \tilde{f}'). \label{limit_Gamma_a_b}
\end{align}
By Lemma \ref{lemma_bensoussan_result}(1), this also implies that $\lim_{b \rightarrow \infty}\Gamma(a,b) = \infty$ if $a > \underline{a}$ and $\lim_{b \rightarrow \infty}\Gamma(a,b) = -\infty$ if $a < \underline{a}$.  Therefore, for fixed $a \in (\underline{a}, \overline{a})$, the infimum $\underline{\Gamma}(a) := \inf_{b \geq a} \Gamma(a,b)$ exists and is increasing in $a$ because the (right-)derivative with respect to $a$ becomes
\begin{align}
\frac \partial {\partial a}\Gamma(a+,b) 
&= -  \tilde{f}'(a+) W^{(q)}(b-a), \quad a < b, \label{derivative_Gamma_a}
\end{align}
which is positive for $a < \overline{a}$, and for any $a' < a < \overline{a}$ (such that $\tilde{f}' < 0$ on $(a', a)$),
\begin{align*}
\underline{\Gamma} (a') &\leq \inf_{b \geq a} \Gamma(a', b) = \inf_{b \geq a} \Big(\Gamma(a,b) + \int_{a'}^a   \tilde{f}'(y) W^{(q)}(b-y) \diff y \Big) \\ &\leq  \inf_{b \geq a} \Big(\Gamma(a,b) + \int_{a'}^{(a+a')/2}   \tilde{f}'(y) W^{(q)}(b-y) \diff y \Big) \\ &\leq  \inf_{b \geq a} \Big(\Gamma(a,b) + W^{(q)}\Big(b-\frac{a+a'} 2 \Big)  \int_{a'}^{(a+a')/2}   \tilde{f}'(y) \diff y \Big) \\ &\leq   \underline{\Gamma}(a) + W^{(q)}\Big(\frac{a-a'} 2 \Big)  \int_{a'}^{(a+a')/2}   \tilde{f}'(y) \diff y < \underline{\Gamma} (a).
\end{align*}
It is also easy to see that the function $\underline{\Gamma}(a)$ is continuous on $(\underline{a}, \overline{a})$.

In view of these arguments, as we decrease the value of $a$ from $\overline{a}$ to $\underline{a}$, there are two scenarios:
\begin{enumerate}
\item The curve $b \mapsto \Gamma(a,b)$ downcrosses the x-axis for a finite $b$ for some $a \in (\underline{a}, \overline{a})$; i.e., there exists $a' \in (\underline{a}, \overline{a})$ such that $\underline{\Gamma}(a') < 0$.
\item The curve $b \mapsto \Gamma(a,b)$ is uniformly positive for any choice of $a \in (\underline{a}, \overline{a})$; i.e., $\underline{\Gamma}(a) \geq 0$ for all $a \in (\underline{a}, \overline{a})$.
\end{enumerate}

For the first scenario, due to the continuity and increasingness of $\underline{\Gamma}$ on $(\underline{a}, \overline{a})$ and because $\underline{\Gamma}(\overline{a}) = C_D + C_U > 0$, there must exist a \emph{unique} $a^* \in (\underline{a}, \overline{a})$ such that $\underline{\Gamma}(a^*) = 0$.  By calling $b^*$ the largest value of the minimizers of $\Gamma(a^*, \cdot)$, we must have that $\Gamma(a^*, b^*) = 0$.  In addition, if the function $\gamma(a^*, \cdot)$ is continuous at $b^*$  then $\gamma(a^*, b^*) = 0$ due to the property of the local minimum.  
Notice, in view of the definition of $\gamma(\cdot, \cdot)$ as in \eqref{small_gamma}, that $\gamma(a,b) < 0$ for any $a<b\leq\overline{a}$, and hence such $b^* > \overline{a}$.


For the second scenario, we have $\Gamma(a, b) \geq 0$ for any $a \in (\underline{a}, \overline{a})$ and $b \geq a$.   Taking $a \downarrow \underline{a}$, we have $\Gamma(\underline{a}, b) \geq 0$ for any $b \geq \underline{a}$.  By \eqref{limit_Gamma_a_b}, we see that $v_{\underline{a}, \infty} (x) := \lim_{b \rightarrow \infty}v_{\underline{a}, b} (x)$ equals $\tilde{v}_{\underline{a}}(x)$ as in \eqref{value_function_no_cost_case}, that is attained by the strategy $\pi_{\underline{a},\infty}$ comprising of the single barrier strategy $U^{\underline{a}, \infty}$ as in \eqref{single_reflected} and $D^\pi = D^{\underline{a}, \infty} \equiv 0$.



%

We summarize the results in the lemma below.
\begin{lemma} \label{lemma_existence}
There exist a unique $a^*$  such that $\Gamma(a^*,x) \geq 0$  for all $x \in [a^*, \infty)$, and $b^*$ (defined as the largest minimizer of $\Gamma(a^*,\cdot)$) such that either \textbf{Case 1} or \textbf{Case 2} defined below holds.
\begin{description}
\item[Case 1] $\underline{a} < a^* < \overline{a} < b^*$ and
\begin{align*}
\Gamma(a^*, b^*) = 0.
\end{align*}
Moreover, if $\gamma(a^*, b^*)$ is continuous  at $b^*$ then we also have that  $\gamma(a^*, b^*) = 0$.
\item[Case 2] $a^* = \underline{a}$ and $b^* = \infty$ and
\begin{align*}
\lim_{b \rightarrow \infty} \frac {\Gamma(\underline{a},b)} {W^{(q)}(b-\underline{a})} = 0.
\end{align*}
\end{description}
\end{lemma}

\begin{remark} \label{remark_continuity_gamma}  In view of \textbf{Case 1} of Lemma \ref{lemma_existence}, the function $b \mapsto \gamma(a^*, b)$ is continuous  at $b^*$ if $f$ is differentiable at $b^*$ or $X$ is of unbounded variation (i.e.\ $W^{(q)}(0) = 0$ as in \eqref{eq:Wqp0}).
\end{remark}

\begin{remark}
\begin{enumerate}
\item  Because $a^* <  \overline{a}$ and $\overline{a} \rightarrow -\infty$ as $C_U \rightarrow \infty$, we must have  $a^* \rightarrow -\infty$ as $C_U \rightarrow \infty$.  On the other hand, $a^*$ must be finite when $C_U$ is finite because $a^* \geq \underline{a}$.
\item 
By \eqref{limit_Gamma_a_b} and Lemma \ref{lemma_bensoussan_result}(1), for any $a > \underline{a}$, $\underline{\Gamma}(a) - C_D - C_U = \inf_{b \geq a}\varphi_a (b;\tilde{f}') > -\infty$.  Because this does not depend on the value of $C_D$, we have $\underline{\Gamma}(a) \rightarrow \infty$ as $C_D \rightarrow \infty$.  Equivalently, for any $a > \underline{a}$, we can choose a sufficiently large $C_D$ so that $\underline{\Gamma}(a) > 0$. Hence, $a^* \rightarrow \underline{a}$ and $b^* \rightarrow \infty$ as $C_D \rightarrow \infty$.  
%
\end{enumerate}
\end{remark}

\section{Verification Lemma} \label{section_verification}
With $(a^*, b^*)$ whose existence is proved in Lemma \ref{lemma_existence}, our candidate value function becomes, by \eqref{expression_v_a_b}, for all $x \leq b^*$,
\begin{align} \label{expression_value_function}
\begin{split}
v_{a^*,b^*} (x) 
&= -C_U R^{(q)}(x-a^*) + \frac {f(a^*)} q Z^{(q)} (x-a^*) - \varphi_{a^*}(x; f) \\
&=  -C_U R^{(q)}(x-a^*) + \frac {f(a^*)} q - \int_{a^*}^x \overline{W}^{(q)}(x-y) f'(y) \diff y.
\end{split}
\end{align}
Integration by parts gives (for more details, see Lemma 4.1 of \cite{Yamazaki_2013})
\begin{align*}
\varphi_{a^*} (x;f) 
&= \varphi_{a^*} (x;\tilde{f}) - C_U \left[  a^* Z^{(q)}(x-a^*) + \overline{Z}^{(q)}(x-a^*) -x\right], \quad x \in \R.
\end{align*}
Hence we can also write
\begin{align} 
v_{a^*,b^*} (x) 
&= -C_U \Big( \frac {\psi'(0+)} q  + x \Big)+ \frac {\tilde{f}(a^*)} q Z^{(q)} (x-a^*) -  \varphi_{a^*} (x;\tilde{f}), \quad x \leq b^*. \label{value_function_simplified}
\end{align}

Let $\mathcal{L}$ be the infinitesimal generator associated with
the process $X$ applied to a sufficiently smooth function $h$
\begin{align*}
\mathcal{L} h(x) &:= c h'(x) + \frac 1 2 \sigma^2 h''(x) + \int_{(-\infty,0)} \left[ h(x+z) - h(x) -  h'(x) z 1_{\{-1 < z < 0\}} \right] \nu(\diff z), \quad x \in \R.
\end{align*}
By Lemma \ref{lemma_smoothfit} and Remarks \ref{remark_smoothness_zero}(1) and \ref{remark_continuity_gamma}, the function $v_{a^*, b^*}$ is $C^1 (\R)$ (resp.\ $C^2(\R)$) when $X$ is of bounded (resp.\ unbounded) variation.  Moreover, the integral part is well defined and finite by Assumption \ref{assump_finiteness_mu} and because $v_{a^*, b^*}$ is linear below $a^*$.
Hence, $\mathcal{L} v_{a^*,b^*} (\cdot)$ makes sense anywhere on $\R$.  

The following theorem addresses some additional conditions that are sufficient for the optimality of $v_{a^*, b^*}$.

\begin{theorem}[Verification lemma] \label{verification_lemma}
Suppose 
\begin{enumerate}
\item $-C_U \leq v_{a^*,b^*}'(x)$ for all $x \in (a^*,b^*)$,
\item $(\mathcal{L}-q)v_{a^*,b^*}(x) + f(x) \geq 0$ for all $x > b^*$.
\end{enumerate}
Then, we have
\begin{align*}
v_{a^*,b^*}(x) = \inf_{\pi \in \Pi}v^\pi (x), \quad x \in \R,
\end{align*}
and $\pi_{a^*,b^*}$ is the optimal strategy.
\end{theorem}

We shall later show that the conditions (1) and (2) of  the above theorem are satisfied if the function $f$ is convex or more generally Assumption \ref{assump_dec_inc} below holds.

In order to show Theorem \ref{verification_lemma} above, we first show Lemmas \ref{verification1} and \ref{lemma_v_prime_bounded_C_D} below.  
\begin{lemma} \label{verification1}
\begin{enumerate}
\item We have $(\mathcal{L}-q)v_{a^*, b^*}(x) + f(x)= 0$ for $a^* < x < b^*$.
\item We have $(\mathcal{L}-q)v_{a^*, b^*}(x) + f(x) \geq 0$ for $ x \leq a^*$.
\end{enumerate}
\end{lemma}
\begin{proof}
As in the proof of  Theorem 2.1 in \cite{Bayraktar_2012},
%
$(\mathcal{L}-q) Z^{(q)}(y-a^*) = (\mathcal{L}-q) R^{(q)}(y-a^*) = 0$ for any $a^* < y < b^*$. 
On the other hand, as in the proof of Lemma 4.5 of \cite{Egami-Yamazaki-2010-1}, $(\mathcal{L}-q) \varphi_{a^*} (x;f) = f(x)$.
Hence in view of  \eqref{expression_value_function},  (1) is proved.

 For (2), by \eqref{value_function_simplified}, $v_{a^*, b^*}(x) =  [-C_U {\psi'(0+)}  + {\tilde{f}(a^*)}] / q- C_U x$, for $x < a^*$, and hence $(\mathcal{L}-q)v_{a^*, b^*}(x) + f(x) = \tilde{f}(x) - \tilde{f}(a^*)$.
This is positive by $x \leq a^* < \overline{a}$ and Assumption \ref{assump_f_g}(2), as desired.
\end{proof}

By \eqref{about_gamma} and \eqref{v_derivative_general},
\begin{align}
v_{a^*,b^*}'(x) = -\Gamma(a^*,x) + C_D, \quad a^* \leq x \leq b^*. \label{derivative_value_function}
\end{align}

\begin{lemma}  \label{lemma_v_prime_bounded_C_D} For all $x \in \R$, we have $v_{a^*,b^*}'(x) \leq C_D$.
\end{lemma}
\begin{proof}
By Lemma \ref{lemma_existence}, we must have $\Gamma(a^*,x) \geq 0$ on $[a^*, b^*]$ and hence  in view of \eqref{derivative_value_function}, this inequality holds for $x \in [a^*, b^*]$.  For $x \in (-\infty, a^*)$, we have  $v_{a^*,b^*}'(x) = - C_U$, which is smaller than $C_D$ by \eqref{assump_C_sum}.  Finally, for $x \in (b^*, \infty)$, we have  $v_{a^*,b^*}'(x) = C_D$.
\end{proof}

We are now ready to give a proof for Theorem \ref{verification_lemma}.

\begin{proof}[Proof of Theorem \ref{verification_lemma}] As a short-hand notation, let $v \equiv v_{a^*,b^*}$ in this proof.  By \eqref{assump_C_sum}, Lemma \ref{lemma_v_prime_bounded_C_D} and the assumption (1), 
\begin{align}
- C_U \leq v'(x) \leq C_D, \quad x \in \R. \label{v_derivative_bounded}
\end{align}

As discussed in the introduction, we can focus on the strategy $\pi \in \Pi$ such that $U_t^\pi$ and $D_t^\pi$ are not increased simultaneously.  Fix any such admissible strategy $\pi \in \Pi$.  Thanks to the smoothness of $v$ described above, It\^o's formula (see, e.g., page 78 of \cite{MR2273672}) gives
 \begin{equation*}
 v(Y_t^{\pi})-v(Y_{0-}^{\pi})= \int_{[0,t]} v'(Y_{s-}^{\pi})\diff Y_s^{\pi}+\frac{\sigma^2}{2}\int_0^t v'' (Y_{s}^{\pi}) \diff s + \sum_{0 \leq s\leq t}[v(Y_s^{\pi})-v(Y_{s-}^{\pi})-v'(Y_{s-}^{\pi})\Delta Y_s^{\pi}].
 \end{equation*}
Define the difference of the control processes $\xi_t^{\pi} := U_t^{\pi} - D_t^{\pi}, \; t \geq 0$.
Then $Y_{t}^\pi = X_t + \xi_t^\pi$ and
\begin{align*}
\Delta \xi_t^\pi = \left\{ \begin{array}{ll} \Delta U_t^{\pi}, & \textrm{if } \Delta \xi_t^\pi \geq 0, \\- \Delta D_t^{\pi}, & \textrm{if }  \Delta \xi_t^\pi < 0, \end{array} \right. \quad \textrm{and}  \quad \diff \xi^{\pi,c}_t = \left\{ \begin{array}{ll} \diff U_t^{\pi,c}, & \textrm{if }  \diff \xi_t^{\pi,c} \geq 0, \\- \diff D_t^{\pi,c}, & \textrm{if } \diff \xi_t^{\pi,c} < 0, \end{array} \right.
\end{align*}
where we denote $\Delta \zeta_t := \zeta_t - \zeta_{t-}$ and $\zeta^c$ as the continuous part of a process $\zeta$.
We have
\begin{align*}
 \int_{[0,t]} v'(Y_{s-}^{\pi})\diff Y_s^{\pi}&=\int_{[0,t]} v'(Y_{s-}^{\pi}) \diff X_s+\int_0^t v'(Y_{s}^{\pi}) \diff \xi_s^{\pi,c} + \sum_{0 \leq s\leq t}v'(Y_{s-}^{\pi})\Delta \xi_s^\pi. 
 \end{align*}

 From the L\'evy-It\^o decomposition theorem (e.g., Theorem 2.1 of \cite{Kyprianou_2006}), we know that
 $$
 X_t=(\sigma B_t +c t)+  \left(\int_{[0,t]}\int_{(-\infty,-1]} y N(\diff s\times \diff y)\right)+\left(\lim_{\varepsilon\downarrow 0}\int_{[0,t]}\int_{(-1,-\varepsilon)}y(N(\diff s\times \diff y)-\nu(\diff y) \diff s)\right),
 $$
 where $B_t$ is a standard Brownian motion and $N$ is a Poisson random measure in   the measurable space  $([0,\infty)\times(-\infty,0),\B [0,\infty)\times \B (-\infty,0), \diff t\times\nu( \diff x)).$ The last term is a square integrable martingale, to which the limit converges uniformly on any compact $[0,T]$.

 Using this decomposition
%
and defining $A_t^\pi:=Y_{t-}^\pi+\Delta X_t$, $t\geq 0$ (so that $A_t^\pi+\Delta \xi_t^\pi=Y_{t}^\pi$), integration by parts gives (see, e.g., the proof of Theorem 3.1 of \cite{Hernandez_Yamazaki_2013} for details),
\begin{align*}
 e^{-q t}v(Y_t^{\pi})-v(Y_{0-}^{\pi}) 
 &= \int_0^t e^{-qs}(\mathcal{L}-q)v(Y_{s}^\pi)\diff s+J_t +M_t,
 \end{align*}
 with
\begin{align*}
J_t &:=\int_0^t e^{-qs}v'(Y_{s}^\pi)\diff U_s^{\pi,c}+\sum_{0\leq s\leq t}e^{-qs}[v(A_{s}^\pi+\Delta U_s^\pi)-v(A_{s}^\pi)] 1_{\{\Delta U_s^\pi> 0\}} \\
&-\int_0^t e^{-qs}v'(Y_{s}^\pi)\diff D_s^{\pi,c}+\sum_{0\leq s\leq t}e^{-qs}[v(A_{s}^\pi-\Delta D_s^\pi)-v(A_{s}^\pi)] 1_{\{\Delta D_s^\pi > 0\}}, \\
 M_t &:= \int_0^t \sigma e^{-qs} v'(Y_{s}^{\pi}) \diff B_s +\lim_{\varepsilon\downarrow 0}\int_{[0,t]} \int_{(-1,-\varepsilon)} e^{-qs}v'(Y_{s-}^{\pi})y (N(\diff s\times \diff y)-\nu(\diff y) \diff s)\\
 &+\int_{[0,t]} \int_{(-\infty,0)}e^{-qs}(v(Y_{s-}^\pi+y)-v(Y_{s-}^\pi)-v'(Y_{s-}^\pi)y1_{\{y\in(-1,0)\}})(N(\diff s\times \diff y)-\nu(\diff y) \diff s).
 \end{align*}
By \eqref{v_derivative_bounded}, we have the inequality: $J_t 
\geq - C_U \int_{[0,t]} e^{-qs}\diff U_s^\pi - C_D \int_{[0,t]} e^{-qs}\diff D_s^\pi$.  Moreover, by  Lemma \ref{verification_lemma}(2) and the assumption (2) of this theorem, $(\mathcal{L}-q)v(x) \geq - f(x)$ for all $x \in \R$.

Let $\tau_n^\pi := \inf \{ t \geq 0: |Y_{t}^\pi| > n\}$, $n > 0$.
Optional sampling gives 
$$
v(x)\leq  \E_x \left[\int_0^{t \wedge \tau_n^\pi}e^{-q s} f(Y_{s}^\pi)\diff s+C_U \int_{[0,t \wedge \tau_n^\pi]}e^{ -qs}\diff  U^\pi_s+C_D \int_{[0,t \wedge \tau_n^\pi]}e^{ -qs}\diff  D^{\pi}_s+ e^{-q (t \wedge \tau_n^\pi)}v(Y_{t \wedge \tau_n^\pi }^\pi)\right].
$$
By \eqref{def_f_tilde}, $\E_x [\int_0^{t \wedge \tau_n^\pi}e^{-q s} f(Y_{s}^\pi)\diff s ] = \E_x [\int_0^{t \wedge \tau_n^\pi}e^{-q s} \tilde{f}(Y_{s}^\pi)\diff s ]  - C_U \E_x [\int_0^{t \wedge \tau_n^\pi} q e^{-q s} Y_{s}^\pi\diff s ]$.
Because $\widetilde{f}$ admits a global minimum at $\overline{a}$ by Assumption \ref{assump_f_g}(2) (and hence is bounded from below),  dominated convergence applied to the negative part of the integrand and monotone convergence for the other part give
\begin{align*}
\lim_{t, n \uparrow \infty}\E_x \Big[\int_0^{t \wedge \tau_n^\pi}e^{-q s} \tilde{f}(Y_{s}^\pi)\diff s \Big] = \E_x \Big[\int_0^{\infty}e^{-q s} \tilde{f}(Y_{s}^\pi)\diff s \Big].
\end{align*}
On the other hand, because $(Y_{s}^\pi)^-  \leq - \underline{X}_s + |x|+ D_s^\pi$ (where we define $\underline{X}_t := \inf_{0 \leq t' \leq t} X_{t'}$, $t \geq 0$, as the running infimum process),
\begin{align} \label{bound_Y_negative_part_integral}
\begin{split}
\int_0^{t \wedge \tau_n^\pi}q e^{-q s} (Y_{s}^\pi)^- \diff s  &\leq  -\int_0^{t \wedge \tau_n^\pi}q e^{-q s} (\underline{X}_{s} - |x| )\diff s + \int_0^{t \wedge \tau_n^\pi}q e^{-q s} D_{s}^\pi\diff s \\
&\leq  -\int_0^{\infty}q e^{-q s} \underline{X}_{s}\diff s +   {|x|} + \int_0^{\infty}q e^{-q s} D_{s}^\pi\diff s.
\end{split}
\end{align}
Notice that 
\begin{align}
- \E \Big[ \int_0^{\infty}q e^{-q s} \underline{X}_{s}\diff s \Big]= \Phi(q)^{-1} - \frac {\psi'(0+)} q \label{expectation_inf_process}
\end{align}
  by the duality and the Wiener-Hopf factorization (see, e.g., the proof of Lemma 4.4 of \cite{Hernandez_Yamazaki_2013}), which is finite by Assumption \ref{assump_finiteness_mu}.


 Integration by parts gives $\E_x [ \int_{[0,t]}e^{ -qs}\diff  D^{\pi}_s ] = \E_x [e^{-qt} D_t^\pi] + \E_x [ \int_0^t q e^{-qs} D_s^\pi \diff s ] \geq \E_x [ \int_0^t q e^{-qs} D_s^\pi \diff s ]$.
Hence, by \eqref{feasibility_L_R},
\begin{align}
\infty > \E_x \left[ \int_{[0,\infty)}e^{ -qs}\diff  D^{\pi}_s \right]  \geq \E_x \left[ \int_0^\infty q e^{-qs} D_s^\pi \diff s \right]. \label{L_cumulative}
\end{align}
This together with \eqref{bound_Y_negative_part_integral} and \eqref{expectation_inf_process} gives $\lim_{t, n \uparrow \infty}\E_x \big[ \int_0^{t \wedge \tau_n^\pi}q e^{-q s} (Y_{s}^\pi)^- \diff s \big] = \E_x \big[ \int_0^{\infty}q e^{-q s} (Y_{s}^\pi)^- \diff s \big] < \infty$. Similar arguments show that $\lim_{t, n \uparrow \infty}\E_x \big[ \int_0^{t \wedge \tau_n^\pi}q e^{-q s} (Y_{s}^\pi)^+ \diff s \big] = \E_x \big[ \int_0^{\infty}q e^{-q s} (Y_{s}^\pi)^+ \diff s \big] < \infty$.

By these and the monotonicity of $D_t^\pi$ and $U_t^\pi$ in $t$, monotone convergence gives  a bound:
\begin{align}
\begin{split}
v(x) = v(U^\pi_{0-}) &\leq  \E_x \left[\int_0^{\infty}e^{-q s} f(Y_{s}^\pi)\diff s+C_U \int_{[0,\infty)}e^{ -qs}\diff  U^\pi_s+C_D \int_{[0,\infty)}e^{ -qs}\diff  D^{\pi}_s \right] \\ &+ \limsup_{t, n \rightarrow \infty} \E_x \left[e^{-q (t \wedge \tau_n^\pi)}v^+(Y_{t \wedge \tau_n^\pi}^\pi)\right]. \label{v_bound_semifinal}
\end{split}
\end{align}

It remains to show that the last term of the right hand side vanishes.  Indeed, we have $\underline{X}_t - D_t^\pi \leq Y_{t \wedge \tau_n^\pi}^\pi \leq \overline{X}_t + U_t^\pi$,  $t \geq 0$ (where we define $\overline{X}_t := \sup_{0 \leq t' \leq t} X_{t'}$, $t \geq 0$, as the running supremum process).
In view of this and \eqref{v_derivative_bounded},  it is sufficient to show that $\E_x [ e^{-qt}(\overline{X}_t + U_t^\pi)]$ and $\E_x [e^{-qt} (-\underline{X}_t + D_t^\pi)]$  vanish in the limit.  

First, as in Lemma 3.3 and Remark 3.2 of \cite{Hernandez_Yamazaki_2013},  $\E_x [e^{-qt}\overline{X}_t]$ and $\E_x [e^{-qt} \underline{X}_t ]$ vanish in the limit as $t \rightarrow \infty$.
On the other hand, by \eqref{L_cumulative} and the monotonicity of $t \mapsto D_t^\pi$,
\begin{align*}
0 = \lim_{t \rightarrow \infty}\E_x \left[ \int_{t}^\infty q e^{-qs} D_s^\pi \diff s \right] \geq \limsup_{t \rightarrow \infty}\E_x \left[ D_{t}^\pi \int_{t }^\infty q e^{-qs}  \diff s \right] = \limsup_{t \rightarrow \infty}\E_x \Big[ e^{-qt} D_{t}^\pi \Big] \geq 0.
\end{align*}
Similarly, $\lim_{t\rightarrow \infty} \E_x [ e^{-q t} U_{t}^\pi ] = 0$ also holds.

These together with \eqref{v_bound_semifinal} show $v(x) \leq v^\pi (x)$ for all $\pi \in \Pi$.  We also have $v(x) \geq \inf_{\pi \in \Pi}v^\pi(x)$ because $v$ is attained by an admissible strategy $\pi_{a^*,b^*} \in \Pi$.  This completes the proof.
\end{proof}

Showing the conditions (1) and (2) of Theorem \ref{verification_lemma} above is the most challenging task of this problem.  However, \textbf{Case 2} 
(i.e.\ $a^* = \underline{a}$ and $b^* = \infty$) can be handled easily; we defer the discussion on \textbf{Case 1} to the next section. 
\begin{theorem}
In \textbf{Case 2}, we have $v_{\underline{a},\infty}(x) = \inf_{\pi \in \Pi}v^\pi (x)$, $x \in \R$,
and $\pi_{\underline{a},\infty}$ is the optimal strategy.
\end{theorem}
\begin{proof}
In view of the conditions for  Theorem \ref{verification_lemma}, we only need to show the condition (1).  By \eqref{derivative_value_function}, $v_{\underline{a}, \infty}'(x) = -\Gamma(\underline{a},x) + C_D =  - C_U - \varphi_{\underline{a}} (x;\tilde{f}')$.
The proof is complete because $\varphi_{\underline{a}} (x;\tilde{f}')$ is nonpositive by the proof of Proposition 7.4 of \cite{Yamazaki_2013}.
\end{proof}

\section{Sufficient condition for optimality for Case 1} \label{section_sufficient_condition} We shall now investigate a sufficient optimality condition for \textbf{Case 1}  so that the assumptions in Theorem \ref{verification_lemma} are satisfied.    Throughout this section, we assume \textbf{Case 1}  and the following. 
\begin{assump} \label{assump_dec_inc}We assume that, for every $a < \overline{a}$, there exists $\tilde{b}(a)  \in (a, \infty]$ such that
\begin{align*}
\gamma(a, b) \leq  (>) 0 \Longleftrightarrow b  < (>) \tilde{b}(a), \quad b \geq a.
\end{align*}
\end{assump}
Equivalently, this assumption says that the function $b \mapsto \Gamma(a, b)$ is first nonincreasing and then  increasing (or nonincreasing monotonically), given $a < \overline{a}$; note from \eqref{small_gamma} that $\gamma(a,b) < 0$ for $a < b < \overline{a}$ and hence the function must first decrease.

As an important condition where Assumption  \ref{assump_dec_inc} holds, we show the following.  It is noted that majority of related control problems assume the convexity of $f$; see, e.g., \cite{dai2013brownian1, dai2013brownian,MR716123}.

\begin{theorem} \label{theorem_convex} If $f$ is convex, then Assumption \ref{assump_dec_inc} holds.
\end{theorem}
\begin{proof}
Fix $a < \overline{a}$.  
Integration by parts applied to \eqref{small_gamma} gives, for all $b > a$,
\begin{align}
\gamma(a,b) = \int_a^b W^{(q)}(b-y) \tilde{f}''(y) \diff y + \tilde{f}'(a+) W^{(q)} (b-a) + \sum_{a < y < b} W^{(q)}(b-y) [\tilde{f}'(y+) - \tilde{f}'(y-)], \label{gamma_in_terms_of_second}
\end{align}
where $\tilde{f}''(y)$ exists a.e.\ on $(a,b)$ by the convexity of $f$.
First, because $\tilde{f}'(x+) - \tilde{f}'(x-) \geq 0$ for any $x \in \R$ by the convexity of $f$, $\gamma(a,b+) - \gamma(a,b-) \geq 0$.

Second, dividing both sides of \eqref{gamma_in_terms_of_second} by  $W^{(q)} (b-a)$ and  taking a derivative with respect to $b$, we have for a.e. $b > a$,
\begin{align*}
\frac \partial {\partial b}\frac {\gamma(a,b)} {W^{(q)}(b-a)} &= \int_a^b \frac \partial {\partial b} \frac {W^{(q)}(b-y)} {W^{(q)}(b-a)} \tilde{f}''(y) \diff y  + \frac {W^{(q)}(0)} {W^{(q)}(b-a)} \tilde{f}''(b-) \\  &+  \sum_{a < y < b} \frac \partial {\partial b} \frac {W^{(q)}(b-y)} {{W^{(q)}(b-a)}} [\tilde{f}'(y+) - \tilde{f}'(y-)]. 
\end{align*}
Here, for any $a < y < b$, the (right) derivative of the fraction ${W^{(q)}(b-y)} / {W^{(q)}(b-a)}$ equals
\begin{align}
\frac {\partial} {\partial b}\frac {W^{(q)}((b-y)+)} {W^{(q)}((b-a)+)}
= \frac {W^{(q)}(b-y)} {W^{(q)}(b-a)} \left[ \frac {W^{(q)'}((b-y)+)} {W^{(q)}(b-y)} -  \frac {W^{(q)'}((b-a)+)} {W^{(q)}(b-a)} \right], \label{fraction_derivative}
\end{align}
which is positive  by Remark  \ref{remark_smoothness_zero}(3).  This together with the convexity of $\tilde{f}$ shows that $b \mapsto {\gamma(a,b)} / {W^{(q)}(b-a)}$ is nondecreasing on $(a, \infty)$.  
By \eqref{Gamma_initial} and Assumption \ref{assump_f_g}(2), we have  $\gamma(a,a+) =  \tilde{f}'(a+) W^{(q)}(0) \leq 0$.  
This means, by the positivity of $W^{(q)}(b-a)$, that ${\gamma(a,\cdot)}$
is first negative and then positive (or uniformly negative).    This completes the proof.
\end{proof}

\begin{lemma} Under Assumption \ref{assump_dec_inc},  the  function $v_{a^*,b^*}$ is convex on $\R$.
\end{lemma}
\begin{proof} Because $a^* < \overline{a}$, Assumption \ref{assump_dec_inc} guarantees that $\tilde{b}(a^*) = b^*$.
Because $b \mapsto \Gamma(a^*, b)$ is nonincreasing on $(a^*,b^*)$, $v_{a^*,b^*}$ is convex on $(a^*,b^*)$ in view of \eqref{derivative_value_function}.  The convexity can be extended to $\R$ by the differentiability at $a^*$ and $b^*$ of $v_{a^*,b^*}$ (if $b^* < \infty$) by Lemma \ref{lemma_smoothfit} and the linearity on $(-\infty, a^*]$ and $[b^*, \infty)$.
\end{proof}

This lemma directly implies the following.

\begin{proposition} \label{proposition_under_the_condition_prop1} Under Assumption \ref{assump_dec_inc}, the condition (1) of Theorem \ref{verification_lemma} holds.
\end{proposition}



Fix any $b \in \R$.  Note that $\Gamma(a,b) > 0$ for any $a \in [\overline{a} \wedge b,b]$ in view of \eqref{about_gamma} and Assumption \ref{assump_f_g}(2). This together with \eqref{derivative_Gamma_a} and Assumption \ref{assump_f_g}(2) (which implies $\lim_{a \downarrow -\infty}\Gamma(a,b) = -\infty$) shows that there exists a unique $a(b) \in (-\infty, \overline{a} \wedge b)$ such that 
\begin{align*}
\Gamma(a(b), b) = 0.
\end{align*}

%
%

\begin{lemma} \label{lemma_monotonicity_b_a} Suppose Assumption \ref{assump_dec_inc} holds.
(i) If $b > b' > b^*$, then $a(b) < a (b') < a^*$ and (ii)  if $b > b^*$, $\gamma(a(b),x-) \geq 0$ for all $x \geq b$.
\end{lemma}
\begin{proof}
We first suppose $b > b^*$ and prove that $a(b) < a^*$. Assume for contradiction that $b > b^*$ and $a(b) \geq a^*$ hold simultaneously.  By    $ a^* \leq  a(b) < \overline{a}$ and \eqref{derivative_Gamma_a}, we have $0 = \Gamma(a(b), b) \geq  \Gamma(a^*, b)$, which is a contradiction because $\Gamma(a^*, b^*) = 0$ and $\Gamma(a^*, \cdot)$ is increasing on $(b^*, b)$ (by Assumption \ref{assump_dec_inc} and because $b^* = \tilde{b}(a^*)$).  Hence whenever $b > b^*$ we must have $a^* > a(b)$.

This also shows $\gamma(a(b),b-) \geq 0$, and hence (ii) by Assumption \ref{assump_dec_inc}.  Indeed, if $\gamma(a(b),b-) < 0$, this means by Assumption \ref{assump_dec_inc} that $\gamma(a(b),\cdot) < 0$ on $(b^*, b)$ and hence $0 = \Gamma(a(b),b) < \Gamma(a(b), b^*)$.  However, this contradicts with $0 = \Gamma(a^*, b^*)$, which is larger than $\Gamma(a(b),b^*)$ by $a(b) < a^* < \overline{a}$ and \eqref{derivative_Gamma_a}.

Now suppose $b > b' > b^*$ and assume for contradiction that $a(b) \geq a(b')$ to complete the proof for (i).  By  \eqref{derivative_Gamma_a} we have $0 = \Gamma(a(b), b) \geq \Gamma(a(b'), b)$, which is a contradiction because $\Gamma(a(b'), \cdot)$ is increasing on $(b', b)$  (due to (ii)) and $\Gamma(a(b'),b') = 0$.
\end{proof}

\begin{lemma} \label{lemma_gamma_b_a_shape} Suppose Assumption \ref{assump_dec_inc} holds. For any $b > b^*$, we have $\Gamma(a(b), y) \leq 0$ for $b^* \leq y \leq b$.
\end{lemma}
\begin{proof}
By definition, $\Gamma(a(b),b)=0$.   Because $a \mapsto \Gamma(a,b^*)$ is increasing on $(-\infty, \overline{a})$ by \eqref{derivative_Gamma_a} and $a(b) < a^* < \overline{a}$ by Lemma \ref{lemma_monotonicity_b_a}(i), we have $\Gamma(a(b), b^*) < \Gamma(a^*, b^*)=0$.

Because $y \mapsto \Gamma(a(b), y)$ is nonincreasing and then increasing on $(b^*, b)$ (or simply monotone on $(b^*, b)$), we must have that
$\Gamma(a(b), y) \leq \max \{ \Gamma(a(b), b^*), \Gamma(a(b), b) \} = 0$ for $b^* \leq y \leq b$.
%
\end{proof}

Using Lemmas \ref{lemma_monotonicity_b_a} and \ref{lemma_gamma_b_a_shape}, we show the second condition of Theorem \ref{verification_lemma}.  Below, we use techniques similar to \cite{Hernandez_Yamazaki_2013, Loeffen_2008}.

\begin{proposition}\label{proof_subharmonic_spec_pos}
Under Assumption \ref{assump_dec_inc}, the condition (2) of Theorem \ref{verification_lemma} holds.
\end{proposition}
\begin{proof}
Fix any $x > b^*$.  
It is sufficient to prove
\begin{align}
(\mathcal{L}-q) (v_{a^*,b^*}-v_{a(x),x})(x-) := \lim_{y \uparrow x}(\mathcal{L}-q) (v_{a^*,b^*}-v_{a(x),x})(y) \geq 0. \label{generator_positive_hypothesis}
\end{align}
Indeed if both \eqref{generator_positive_hypothesis} and $(\mathcal{L}-q) v_{a^*,b^*}(x) + f(x)< 0$ hold simultaneously, then
\begin{align*}
0 > (\mathcal{L}-q) v_{a^*,b^*}(x) + f(x) \geq (\mathcal{L}-q) v_{a(x),x}(x-) + f(x),
\end{align*}
which leads to a contradiction because $(\mathcal{L}-q) v_{a(x),x}(y) + f(y) = 0$ for $a(x) < y < x$ that holds similarly to Lemma \ref{verification1}(1).  Notice that the function $v_{a(x),x}$ admits the same form as \eqref{expression_value_function} (with $a^*$ replaced with $a(x)$) because $\Gamma(a(x), x) = 0$.

Notice from \eqref{smooth_fit_all_b} that both $v_{a^*,b^*}$ and $v_{a(x),x}$ are differentiable on $\R$. 
Similarly to \eqref{derivative_value_function},  \label{generator_parts}
\begin{align}
v_{a(x),x}'(y)
=  -\Gamma(a(x),y)+C_D \quad \textrm{and} \quad
v_{a(x), x}''(y)  =  -\gamma(a(x),y), \quad a(x) < y < x. \label{derivatives_x_b_x}
\end{align}
The dominated convergence theorem gives
\begin{align*}
&(\mathcal{L}-q) (v_{a^*,b^*}-v_{a(x),x})(x-) = c (v_{a^*,b^*}' - v_{a(x),x}')(x) + \frac 1 2 \sigma^2 (v_{a^*,b^*}'' - v_{a(x),x}'')(x-) \\ &+ \int_{(-\infty, 0)} \left[(v_{a^*,b^*}-v_{a(x),x})(x+z) - (v_{a^*,b^*}-v_{a(x),x})(x) -  (v_{a^*,b^*}'- v_{a(x),x}')(x) z 1_{\{-1 < z < 0\}} \right] \nu(\diff z) \\ &- q(v_{a^*,b^*}-v_{a(x),x})(x).
\end{align*}
By the differentiability of $v_{a(x),x}$ and $v_{a^*,b^*}'(x) = C_D$ and $v_{a^*,b^*}''(x) = 0$ as $x > b^*$, this is simplified to
\begin{multline} \label{generator_parts}
(\mathcal{L}-q) (v_{a^*,b^*}-v_{a(x),x})(x-) = - \frac 1 2 \sigma^2  v_{a(x),x}''(x-) \\ + \int_{(-\infty, 0)} \left[(v_{a^*,b^*}-v_{a(x),x})(x+z) - (v_{a^*,b^*}-v_{a(x),x})(x)  \right] \nu(\diff z) - q(v_{a^*,b^*}-v_{a(x),x})(x).
\end{multline}
By taking limits in \eqref{derivatives_x_b_x} and by Lemma \ref{lemma_monotonicity_b_a}(ii),
\begin{align*}
v_{a(x), x}''(x-)  =  -\gamma(a(x),x-) \leq 0.
\end{align*}

In order to prove the positivity of the integral part of \eqref{generator_parts}, we shall prove that
\begin{align}
v_{a(x),x}'(y) \geq  v_{a^*,b^*}'(y), \quad y \in (-\infty,x). \label{derivative_dominant}
\end{align}
 Recall also that $a(x) < a^*$ by Lemma \ref{lemma_monotonicity_b_a}(i). 

(i) For $b^* \leq y < x $, by Lemma \ref{lemma_gamma_b_a_shape}, $v_{a(x),x}'(y)
= -\Gamma(a(x),y)+C_D  \geq C_D = v_{a^*,b^*}'(y)$.

(ii) For $a^* \leq y < b^*$,
\begin{align*}
\begin{split}
v_{a(x),x}'(y)
&= -\Gamma(a(x),y)+C_D \geq  -\Gamma(a^*,y)+C_D  = v_{a^*,b^*}'(y).
\end{split}
\end{align*}
Here the inequality holds because  $a(x) <  a^* < \overline{a}$ and $\Gamma(\cdot, y)$ is increasing by \eqref{derivative_Gamma_a}.

(iii) For $a(x) \leq y < a^*$, by Assumption \ref{assump_dec_inc}, \eqref{Gamma_initial}  and \eqref{derivative_Gamma_a}, $\Gamma(a(x),y) \leq \Gamma(a(x),a(x)) \vee \Gamma(a(x), a^*)  \leq \Gamma(a(x),a(x)) \vee \Gamma(a^*,a^*) =  C_D + C_U$. Hence
\begin{align*}
\begin{split}
v_{a(x),x}'(y)
= -\Gamma(a(x), y)+C_D  \geq -C_U = v_{a^*,b^*}'(y).
\end{split}
\end{align*}

(vi) For $y < a(x)$, we have that $v_{a(x),x}'(y) = v_{a^*,b^*}'(y)=-C_U$.  Hence \eqref{derivative_dominant} holds, and consequently the integral of \eqref{generator_parts} is positive.

Finally, we shall show that 
\begin{align}
v_{a(x),x}(x) \geq  v_{a^*,b^*}(x). \label{inequality_at_x}
\end{align}
By \eqref{value_function_simplified} (which also holds when $a^*$ is replaced with $a(x)$) and $a(x) < a^*$,
\begin{align*} 
v_{a(x),x} (a(x)) 
&= -C_U \frac {\psi'(0+)} q + \frac {\tilde{f}(a(x))} q  - C_U a(x), \\
v_{a^*,b^*} (a(x)) &= -C_U \frac {\psi'(0+)} q + \frac {\tilde{f}(a^*)} q - C_U a(x).
\end{align*}
Because $a(x) < a^* < \overline{a}$, we have $v_{a(x),x} (a(x))  \geq v_{a^*,b^*} (a(x))$ by Assumption \ref{assump_f_g}(2).  This together with \eqref{derivative_dominant} shows \eqref{inequality_at_x}.
Putting altogether, \eqref{generator_positive_hypothesis} indeed holds.  This completes the proof.
\end{proof}

Combining Propositions \ref{proposition_under_the_condition_prop1} and \ref{proof_subharmonic_spec_pos}, we have the following.

\begin{theorem} \label{theorem_for_convex_case} Under Assumption \ref{assump_dec_inc}, we have $v_{a^*,b^*}(x) = \inf_{\pi \in \Pi}v^\pi (x)$ for $x \in \R,$
and $\pi_{a^*,b^*}$ is the optimal strategy.
\end{theorem}

\section{Examples} \label{section_examples}

Recall from Theorems \ref{theorem_convex}  and  \ref{theorem_for_convex_case} that  whenever the running cost function $f$ is convex, the optimality of $v_{a^*, b^*}$ holds.  In this section we consider the following two special cases and study the criteria for \textbf{Case 1} and \textbf{Case 2} as in Lemma \ref{lemma_existence}. 
\subsection{Quadratic case} \label{example_quadratic}
Suppose the running cost function is $f \equiv f_Q$ where
\begin{align}
f_Q(x) := \alpha^- x^2 1_{\{x < 0\}} + \alpha^+ x^2 1_{\{x \geq 0\}}, \quad x \in \R, \label{def_f_Q}
\end{align}
 for some $\alpha^-, \alpha^+ > 0$. The quadratic cost function of this form is used in, e.g., \cite{baccarin2002optimal}.
 
We then have $f_Q'(x) =   2 [\alpha^- 1_{\{ x < 0\}} + \alpha^+ 1_{\{ x \geq 0\}}]  x$ and
$\tilde{f}_Q'(x) =  2 [\alpha^- 1_{\{ x < 0\}} + \alpha^+ 1_{\{ x \geq 0\}}]  x+ q C_U$.
Hence, Assumption \ref{assump_f_g}(2) holds with $\overline{a} = - q C_U / (2 \alpha^-)$ if $C_U \geq 0$ and  $\overline{a} = - q C_U / (2 \alpha^+)$ if $C_U < 0$.  Moreover, for $a \leq 0$,
\begin{align*}
\Psi(a; \tilde{f}') = \left\{ \begin{array}{ll}\frac {2 \alpha^- a + q C_U } {\Phi(q)} + \frac {2 (\alpha^+ - \alpha^-) e^{\Phi(q) a} + 2 \alpha^-} {\Phi(q)^2},  & a \leq 0, \\
\frac {2 \alpha^+ a + q C_U } {\Phi(q)} + \frac {2 \alpha^+} {\Phi(q)^2},
 & a > 0, \end{array} \right.
\end{align*}
and hence $\underline{a}$ is either $\underline{a} \leq 0$ satisfying
\begin{align} \label{a_bar_equation_quadratic}
  {2 \alpha^- \underline{a} + q C_U}  = - \frac 2 {\Phi(q)}  \Big( {(\alpha^+ - \alpha^-) e^{\Phi(q) \underline{a}}}   + {\alpha^-}  \Big),
\end{align}
or $\underline{a} > 0$ satisfying
\begin{align}2 \alpha^+ \underline{a} + q C_U = - \frac 2 {\Phi(q)} \alpha^+. \label{a_bar_equation_quadratic_positive}
\end{align}

In addition, direct computation gives for $a \leq 0$
\begin{align*}
\Gamma(a,b) = C_U + C_D + (2 \alpha^- a + q C_U) \overline{W}^{(q)} (b-a)  + 2 \alpha^- \int_{a}^{0 \wedge b} \overline{W}^{(q)} (b-y) \diff y  + 2 \alpha^+ \int_{0 \wedge b}^{b} \overline{W}^{(q)} (b-y) \diff y, 
\end{align*}
and for $a > 0$,
\begin{align*}
\Gamma(a,b) = C_U + C_D + (2 \alpha^+ a + q C_U) \overline{W}^{(q)} (b-a)+ 2 \alpha^+ \int_{a}^{b} \overline{W}^{(q)} (b-y) \diff y. 
\end{align*}

We first show the following.
\begin{lemma} \label{lemma_W_difference}For any $a \leq 0$, we have  $\lim_{b \uparrow \infty} [  \overline{W}^{(q)}(b) -  \overline{W}^{(q)}(b-a) e^{\Phi(q) a} ]  =  (e^{\Phi(q) a} -1)/q$.
\end{lemma}
\begin{proof}
The case $a = 0$ holds trivially and hence we assume $a < 0$.
By \eqref{two_sided_exit},  $\E_{b} \big[ e^{-q \tau_0^-} 1_{\left\{ \tau_{b-a}^+ > \tau_0^- \right\}} \big] = 1 + q \overline{W}^{(q)}(b) -  (1+ q\overline{W}^{(q)}(b-a))  {W^{(q)}(b)}  / {W^{(q)}(b-a)}$.
 Because this converges to $0$ as $b \uparrow \infty$, we have the convergence:
 \begin{align*}
-\frac  1 q
&= \lim_{b \uparrow \infty} \Big[  \overline{W}^{(q)}(b) -  e^{\Phi(q) a} \overline{W}^{(q)}(b-a) - \frac 1 q \frac {W^{(q)}(b)}  {W^{(q)}(b-a)} +    \frac {\overline{W}^{(q)}(b-a)} {{W^{(q)}(b-a)} } \Big( e^{\Phi(q) a} {W^{(q)}(b-a)}  -   {W^{(q)}(b)}   \Big) \Big]. 
\end{align*}
By equation (95) of Kuznetsov et al.\ \cite{Kuznetsov2013}, we can decompose the scale function so that $W^{(q)}(y) = e^{\Phi(q) y}/\psi'(\Phi(q)) - \hat{u}^{(q)} (y)$, $y \geq 0$, where $\hat{u}^{(q)} (z)$ is uniformly bounded and vanishes as $z \uparrow \infty$.  Hence,  $e^{\Phi(q) a} W^{(q)}(b-a)  -  {W^{(q)}(b)}  =  - \hat{u}^{(q)} (b-a) e^{\Phi(q) a} + \hat{u}^{(q)}(b) \xrightarrow{b \uparrow \infty} 0$.  On the other hand, by \eqref{scale_function_version} and \eqref{W_q_limit}, $W^{(q)}(b)/W^{(q)}(b-a) \xrightarrow{b \uparrow \infty} e^{\Phi(q) a}$ and ${\overline{W}^{(q)}(b-a)} / {W^{(q)}(b-a)} \xrightarrow{b \uparrow \infty} \Phi(q)^{-1}$.  This shows the claim. 
\end{proof}

\begin{proposition} \label{proposition_f_Q}Suppose  $f \equiv f_Q$ as in \eqref{def_f_Q}.   Then,  \textbf{Case 1}  always holds.
\end{proposition}
\begin{proof}
It is sufficient to show that $\Gamma(\underline{a}, b) \xrightarrow{b \uparrow \infty} -\infty$. Indeed, by the continuity of $\Gamma(a,b)$ in $a$, this means that there must exist $a' > \underline{a}$ such that $b \mapsto \Gamma(a',b)$ downcrosses the x-axis, which means \textbf{Case 1}.

(i) We shall first consider the case $\underline{a} \leq 0$. For $b > 0$, by \eqref{a_bar_equation_quadratic}, \begin{align*}
\Gamma(\underline{a},b) 
&= C_U + C_D + (2 \alpha^- \underline{a} + q C_U) \overline{W}^{(q)} (b-\underline{a})  + 2 \alpha^- \int_0^{b-\underline{a}} \overline{W}^{(q)} (y) \diff y  + 2 (\alpha^+ - \alpha^-) \int_0^{b} \overline{W}^{(q)} (y) \diff y  \\
&= C_U + C_D + (2 \alpha^- \underline{a} + q C_U) \Big( \overline{W}^{(q)} (b-\underline{a})  - \Phi(q) \int_0^{b-\underline{a}} \overline{W}^{(q)} (y) \diff y \Big)  \\ &-  2 (\alpha^+ - \alpha^-)  \Big(  {e^{\Phi(q) \underline{a}}}  \int_0^{b-\underline{a}} \overline{W}^{(q)} (y) \diff y  -  \int_0^{b} \overline{W}^{(q)} (y) \diff y \Big).
\end{align*}

By Lemma 7.1 of  \cite{Yamazaki_2013},  for the process $U^{\underline{a}, \infty}$ as defined in \eqref{single_reflected},
we have 
\begin{align}
\E_x \Big[ \int_{[0,\infty)} e^{-qt} \diff U^{\underline{a}, \infty}_{t} \Big]  = - (x-\underline{a})  +  \frac q {\Phi(q)}\Big( {\overline{W}^{(q)} (x-\underline{a})}-  \Phi(q) \int_0^{x-\underline{a}} \overline{W}^{(q)}(y) \diff y \Big) + \frac 1 {\Phi(q)}  - \frac {\psi'(0+)} q. \label{U_for_limit}
\end{align}
For any $x > 0 (\geq \underline{a})$, because $0 \leq U_t^{\underline{a}, \infty} \leq -\underline{X}_{t} \vee 0$, we have $0 \leq \int_{[0,\infty)} e^{-qt} \diff U^{\underline{a}, \infty}_{t} \leq \int_0^\infty q e^{-qt} U^{\underline{a}, \infty}_{t} \diff t \leq  \int_0^\infty q e^{-qt} [(-\underline{X}_t) \vee 0] \diff t$.
Moreover, 
\begin{align*}
\limsup_{x \rightarrow \infty}\E_x \Big[ \int_0^\infty q e^{-qt} [(-\underline{X}_t) \vee 0] \diff t \Big]= \limsup_{x \rightarrow \infty} \E \Big[ \int_0^\infty q e^{-qt} [(-(\underline{X}_t + x)) \vee 0] \diff t \Big] = 0,
\end{align*}
where the last equality holds by dominated convergence by noting that, under $\p$, $0 \leq \int_0^\infty q e^{-qt} [(-(\underline{X}_t + x)) \vee 0 ] \diff t \leq \int_0^\infty q e^{-qt} (-(\underline{X}_t)) \diff t$, which is integrable by  \eqref{expectation_inf_process}.   Hence, \eqref{U_for_limit} vanishes as  $x \rightarrow \infty$.  Therefore, 
\begin{align} \label{W_overline_difference_diverge}
\overline{W}^{(q)} (b-\underline{a})  - \Phi(q) \int_0^{b-\underline{a}} \overline{W}^{(q)} (y) \diff y \sim \frac {\Phi(q)} q (b-\underline{a}),
\end{align} 
 where $x \sim y$ means $x/y \rightarrow 1$ as $b \rightarrow \infty$.
On the other hand, by l'H\^opital's rule and Lemma \ref{lemma_W_difference}, 
\begin{align*}
 {e^{\Phi(q) \underline{a}}}   \int_0^{b-\underline{a}} \overline{W}^{(q)} (y) \diff y  -  \int_0^{b} \overline{W}^{(q)} (y) \diff y  \sim  b \Big( {e^{\Phi(q) \underline{a}}}  \overline{W}^{(q)} (b-\underline{a})   -  \overline{W}^{(q)} (b)  \Big) \sim b  \frac {1-e^{\Phi(q) \underline{a}}} q.
\end{align*}
Combining these and by \eqref{a_bar_equation_quadratic},
\begin{align*}
\lim_{b \uparrow \infty} \frac {\Gamma(\underline{a}, b)} b = (2 \alpha^- \underline{a} + q C_U)  \frac {\Phi(q)} q  -  2 (\alpha^+ - \alpha^-) \frac {1-e^{\Phi(q) \underline{a}}} q = -\frac {2 \alpha^+} q < 0.
\end{align*}
 This completes the proof for the case $\underline{a} \leq 0$.
 
 (ii) For the case $\underline{a} > 0$, by \eqref{a_bar_equation_quadratic_positive},
 \begin{align*}
\Gamma(a,b) = C_U + C_D - \frac {2 \alpha^+}{\Phi(q)}  \Big( \overline{W}^{(q)} (b-a) - \Phi(q) \int_{a}^{b} \overline{W}^{(q)} (b-y) \diff y \Big),
\end{align*}
which goes to $-\infty$ by \eqref{W_overline_difference_diverge}. This completes the proof.

\end{proof}

\subsection{Linear case} \label{example_linear}
Suppose the running cost function is $f \equiv f_L$ where
\begin{align}
f_L(x) := \alpha^- |x| 1_{\{x \leq 0\}} + \alpha^+ x1_{\{x > 0\}}, \quad x \in \R,  \label{def_f_L}
\end{align}
 for some $\alpha^+, \alpha^-\in \R$.  This linear cost is specifically assumed in related control problems such as \cite{constantinides1978existence, sulem1986solvable}.
For any $x \neq 0$,  $f_L'(x) =  -\alpha^- 1_{\{x < 0\}} + \alpha^+ 1_{\{x > 0\}}$  and
$\tilde{f}_L'(x) = -\alpha^- 1_{\{x < 0\}} + \alpha^+ 1_{\{x > 0\}} + q C_U$.
Hence, in order for Assumption \ref{assump_f_g} to be satisfied, we need to require that
\begin{align}
q C_U + \alpha^+ > 0 > q C_U - \alpha^-. \label{linear_case_assump}
\end{align}
Under \eqref{linear_case_assump}, we have $\overline{a} = 0$, and $\underline{a} < 0$ is the unique value such that 
\begin{align}
 (\alpha^-+\alpha^+)e^{\Phi(q) \underline{a}}  = \alpha^- -q C_U, \label{a_underline_linear}
 \end{align}
 which always exists because \eqref{linear_case_assump} guarantees that $\alpha^-+\alpha^+ > 0$ and $\alpha^- -q C_U \in (0, \alpha^- + \alpha^+)$.
In addition, direct computation gives, for any $a \in [\underline{a}, \overline{a}]$ and $b > a$,
\begin{align*}
\Gamma(a,b) &= C_U + C_D + ( q C_U- \alpha^- ) \overline{W}^{(q)} (b-a)  + ( \alpha^- + \alpha^+)\overline{W}^{(q)} (b), \\
\gamma(a,b) 
&=(q C_U - \alpha^-) W^{(q)}(b-a) + (\alpha^- + \alpha^+ ) W^{(q)}(b).
\end{align*}

\begin{proposition}  \label{proposition_linear} Suppose  $f \equiv f_L$ as in \eqref{def_f_Q} such that \eqref{linear_case_assump} holds.   Then, \textbf{Case 1} holds if $C_D < \alpha^+/q$ and  \textbf{Case 2} holds otherwise.
\end{proposition}
\begin{proof}
Because $\underline{a} < \overline{a} = 0$,
the fraction $ {W^{(q)}(b)} / {W^{(q)}(b-\underline{a})}$ is increasing in $b$ (see \eqref{fraction_derivative}). Hence, 
for any $\underline{a} < b$, by \eqref{a_underline_linear} (together with  \eqref{scale_function_version} and \eqref{W_q_limit}),
\begin{align*}
\frac {\gamma(\underline{a},b)} {W^{(q)}(b-\underline{a})} \nearrow (\alpha^-+\alpha^+) \lim_{b \rightarrow \infty} \frac {W^{(q)}(b)} {W^{(q)}(b-\underline{a})} + (q C_U - \alpha^-) = (\alpha^-+\alpha^+) e^{\Phi(q) \underline{a}} + (q C_U - \alpha^-)  = 0.
\end{align*}
Hence,  $\gamma(\underline{a},b) \leq 0$ uniformly in $b$, or equivalently the function $b \mapsto \Gamma(\underline{a}, b)$ is uniformly decreasing and $\inf_{b \geq \underline{a}}\Gamma(\underline{a}, b) = \lim_{b \rightarrow \infty} \Gamma(\underline{a}, b)$.  By this, \textbf{Case 1} holds if $\lim_{b \rightarrow \infty} \Gamma(\underline{a}, b) < 0$ while \textbf{Case 2} holds if $\lim_{b \rightarrow \infty} \Gamma(\underline{a}, b) \geq 0$ .
Now the proof is complete because, by Lemma \ref{lemma_W_difference} and \eqref{a_underline_linear},
\begin{align*}
\lim_{b \rightarrow \infty}\Gamma(\underline{a},b) 
&= C_U + C_D +  (\alpha^-+\alpha^+)  \frac {e^{\Phi(q) \underline{a}} -1} q \\
&= C_U + C_D   - (\alpha^-+\alpha^+)/q +  \alpha^-/ q -C_U = C_D - \alpha^+/q.
\end{align*}
\end{proof}



\section{Numerical Results} \label{section_numerics}

In this section, we conduct numerical experiments using the spectrally negative \lev process in
 the $\beta$-family introduced by \cite{Kuznetsov_2010_2}. The following definition is due to Definition 4 of \cite{Kuznetsov_2010_2}.
\begin{definition}
A spectrally negative \lev process is said to be in the $\beta$-family if \eqref{laplace_spectrally_positive} is written
\begin{align*}
\psi(z) = \hat{\delta} z + \frac 1 2 \sigma^2 z^2 + \frac \varpi \beta \left\{ B(\alpha + \frac z \beta, 1 - \lambda) - B(\alpha, 1 - \lambda) \right\}
\end{align*}
for some $\hat{\delta} \in \R$, $\alpha > 0$, $\beta > 0$, $\varpi \geq 0$, $\lambda \in (0,3) \backslash \{1,2\}$ and the beta function $B(x,y)$. 

\end{definition}
This is a subclass of the meromorphic \lev process \cite{Kuznetsov_2010} and hence the \lev measure can be written
\begin{align*}
\nu(\diff z) = \sum_{j=1}^\infty p_j \eta_j e^{- \eta_j
|z|} 1_{\{z < 0\}} \diff z,  \quad z
\in \mathbb{R},
\end{align*}
for some $\{ p_k, \eta_k; k \geq 1 \}$.
The equation $\psi(\cdot)=q$  has countably many roots $\{ -\xi_{k,q}; k \geq 1 \}$ that are all negative real numbers and satisfy the interlacing condition: $\cdots < -\eta_k < -\xi_{k,q} < \cdots < -\eta_2 < -\xi_{2,q} < -\eta_1 < -\xi_{1,q} < 0$.
By using similar arguments as in \cite{Egami_Yamazaki_2010_2}, the scale function can be written as 
\begin{align*} 
W^{(q)}(x)   = \frac {e^{\Phi(q) x}} {\psi'(\Phi(q))} - \sum_{i=1}^\infty B_{i,q}e^{-\xi_{i,q} x}, \quad x \geq 0,
\end{align*}
where
\begin{align*}
B_{i,q} := \frac {\Phi(q)} q \frac {\xi_{i,q}A_{i,q}} {\Phi(q)+\xi_{i,q}} \quad \textrm{and} \quad 
A_{i,q} :=  \left( 1 - \frac {\xi_{i,q}} {\eta_i} \right) \prod_{j \neq i} \frac {1 - \frac {\xi_{i,q}} {\eta_j}} {1 - \frac {\xi_{i,q}} {\xi_{j,q}}}, \quad i \geq 1.
\end{align*}

Throughout this section, we suppose $\hat{\delta} = 0.1$, $\lambda = 1.5$, $\alpha=3$, $\beta=1$, $\varpi = 0.1$ and $\sigma = 0.2$.  This means that the process is of unbounded variation with jumps of infinite activity.  We also let $q = 0.03$.

\subsection{Quadratic case}  
We shall first consider the quadratic case with $f \equiv f_Q$ as in \eqref{def_f_Q} with $\alpha^- = \alpha^+ = 1$.   As is discussed in Proposition \ref{proposition_f_Q}, \textbf{Case 1} necessarily holds and $a^*$ must be in $(\underline{a}, \overline{a})$.  Due to the fact that $\underline{\Gamma} (\cdot)$ is monotone and $\gamma(a, \cdot)$ satisfies Assumption \ref{assump_dec_inc} by Theorem \ref{theorem_convex}, the values of  $a^*$ and  $b^*$ can be easily computed by a (nested) bisection-type method.

Figure \ref{Gamma_plot_quadratic} shows the functions $b \mapsto \Gamma(a, b)$ and $b \mapsto \gamma(a, b)$ for $a = \underline{a}, (\underline{a} + a^*)/2, a^*, (a^* + \overline{a})/2, \overline{a}$ with the common values $C_U = C_D = 10$.  As has been studied in Subsection \ref{subsection_existence_a_b}, $\Gamma(\overline{a}, \cdot)$ starts at a positive value and increases monotonically.  On the other hand, as in the proof of Proposition \ref{proposition_f_Q}, $\Gamma(\underline{a}, \cdot)$ goes to $-\infty$ (and hence \textbf{Case 1} always holds).  The desired value of $a^*$ is such that the function is tangent (at $b^*$) to the x-axis.  It can be confirmed by the graph on the right that Assumption \ref{assump_dec_inc} indeed holds for each $a$.

 \begin{figure}[htbp]
\begin{center}
\begin{minipage}{1.0\textwidth}
\centering
\begin{tabular}{cc}
 \includegraphics[scale=0.58]{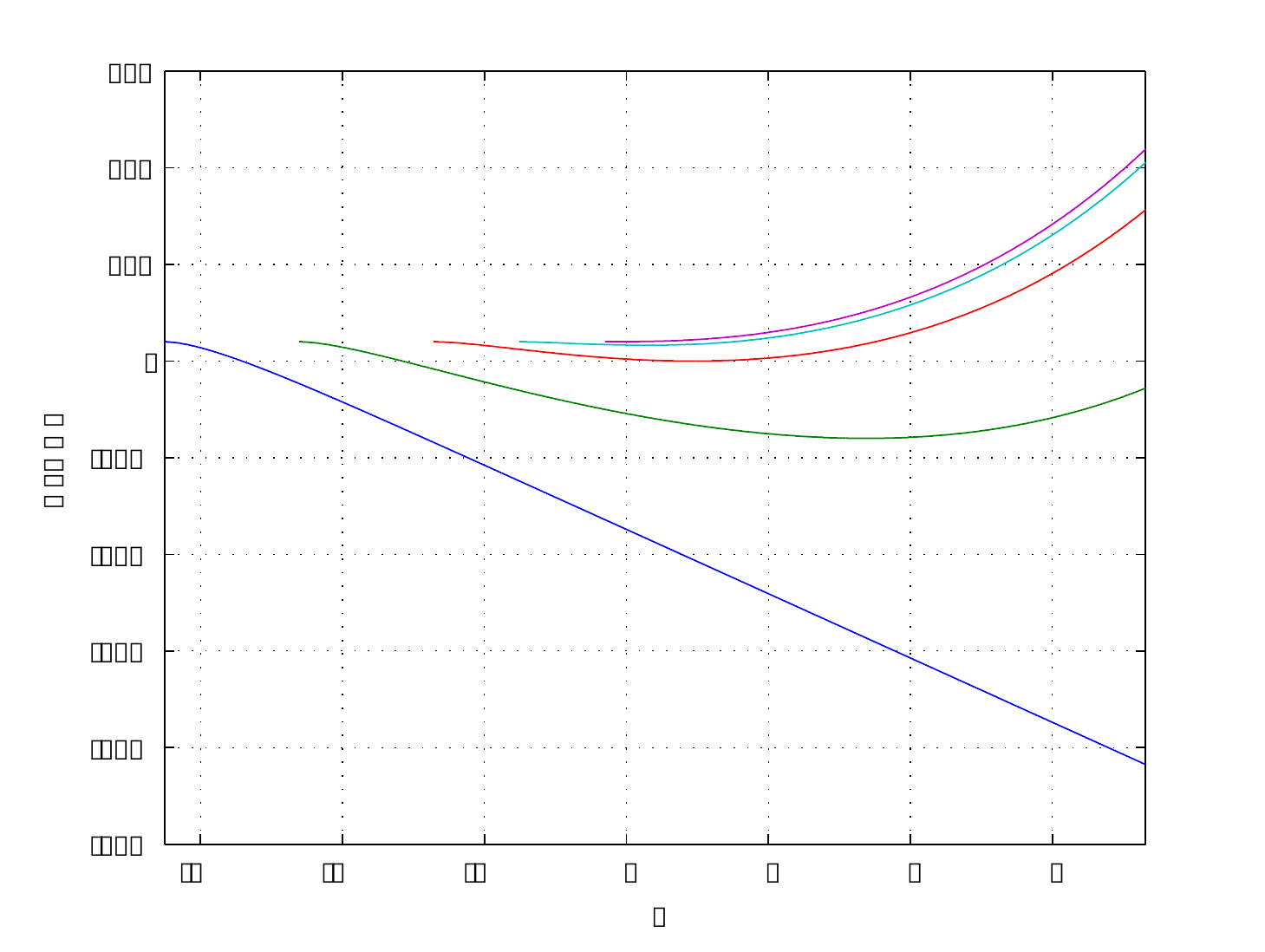} & \includegraphics[scale=0.58]{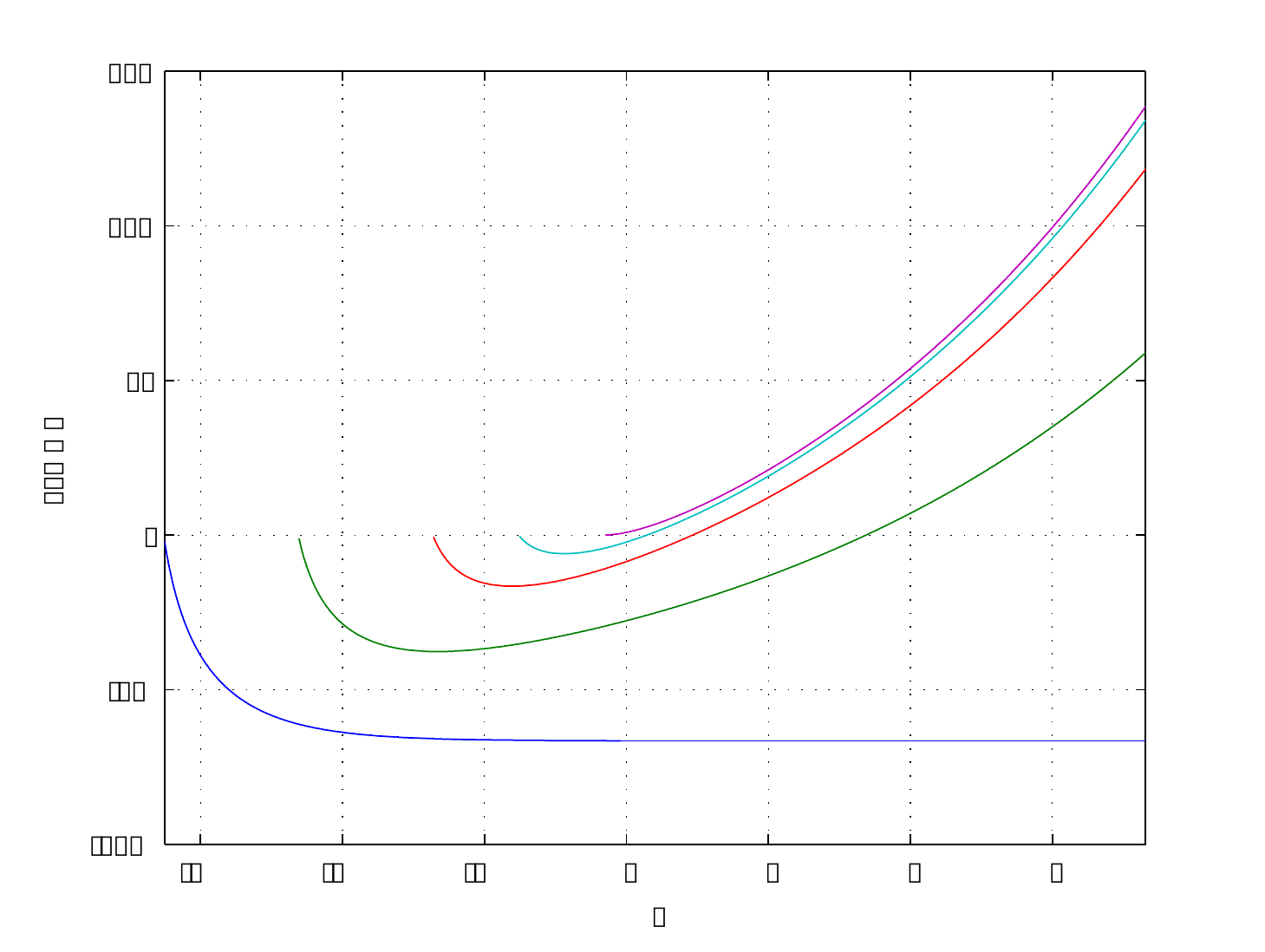}  \\
$\Gamma(a, \cdot)$ & $\gamma(a, \cdot)$
\end{tabular}
\end{minipage}
\caption{Quadratic case: the plots of $b \mapsto \Gamma(a, b)$ (left) and $b \mapsto \gamma(a,b)$ (right) for $a = \underline{a}, (\underline{a} + a^*)/2, a^*, (a^* + \overline{a})/2, \overline{a}$.  The line in red corresponds to the one for $a = a^*$; the point at which $\Gamma(a^*, \cdot)$ is tangent to the x-axis (or $\gamma(a^*, \cdot)$ vanishes) becomes $b^*$.}  \label{Gamma_plot_quadratic}
\end{center}
\end{figure}

In Figure \ref{value_function_plot_quadratic}, we show the value functions for  $C_U \in  E := \{35, 30, 25, 20, 15, 10, 5, 0, -5 \}$ with the common value of $C_D = 6$ (left) and also those for $C_D \in E$ with $C_U= 6$ (right).  With these selections of parameters, \eqref{assump_C_sum} is always satisfied. It is clear that the value function is monotone in both $C_U$ and $C_D$.   The distance between $a^*$ and $b^*$ tends to shrink as $C_U + C_D$ decreases.  In all cases, we can confirm that the optimal barrier levels $(a^*, b^*)$ are indeed finite.

 \begin{figure}[htbp]
\begin{center}
\begin{minipage}{1.0\textwidth}
\centering
\begin{tabular}{cc}
 \includegraphics[scale=0.58]{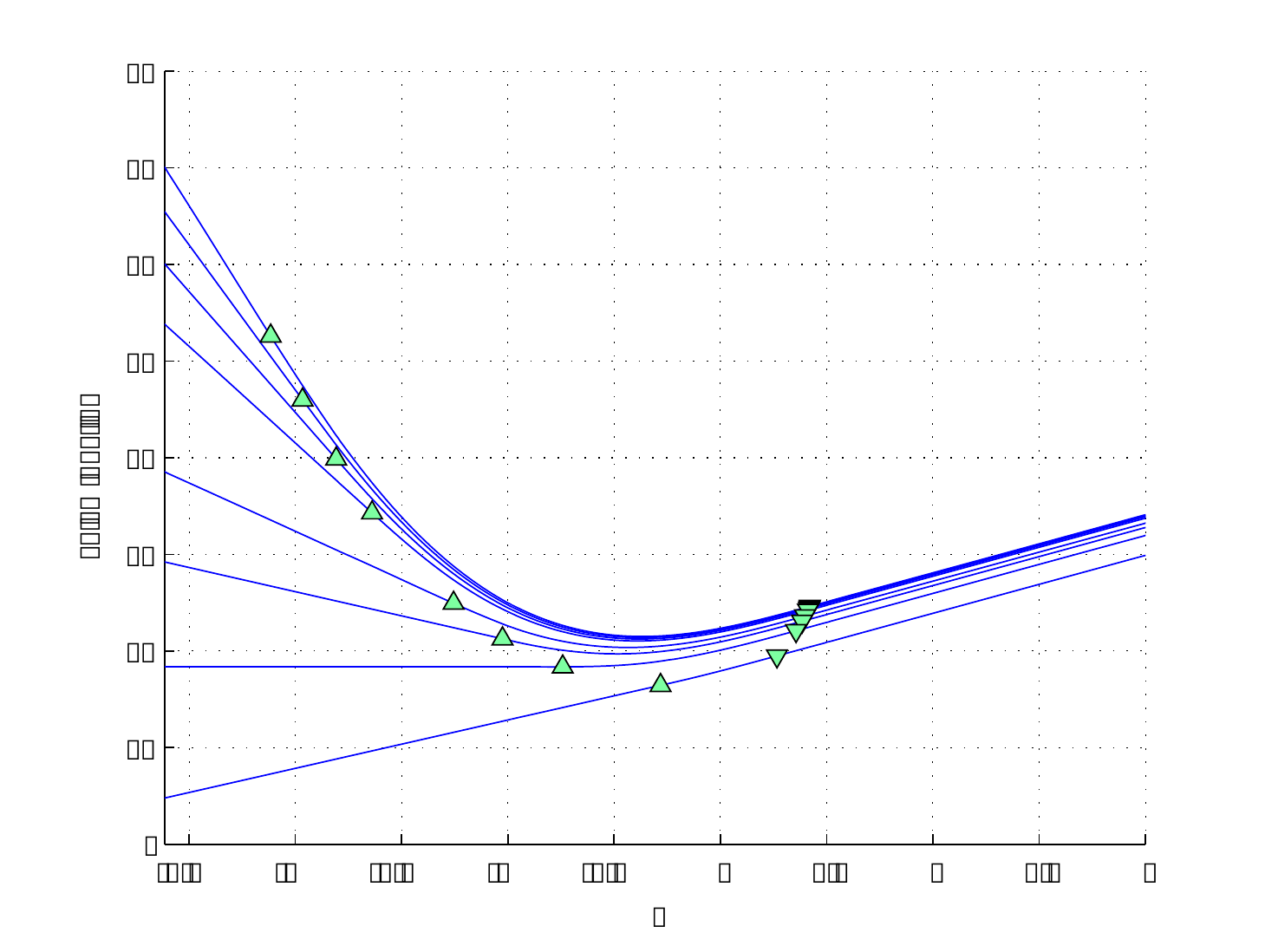} & \includegraphics[scale=0.58]{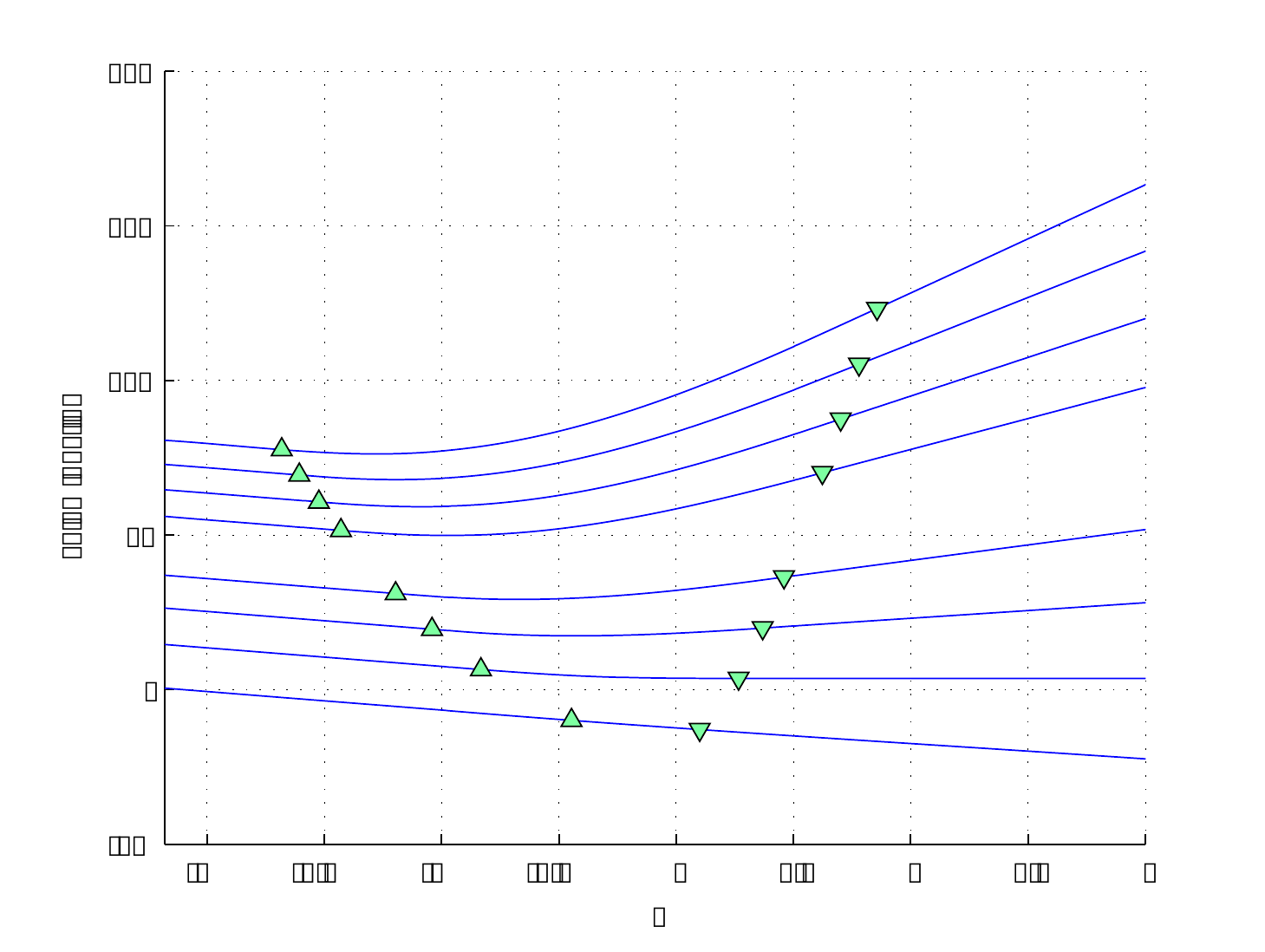}  \\
with respect to $C_U$ & with respect to $C_D$
\end{tabular}
\end{minipage}
\caption{Quadratic case: the plots of the value function for  $C_U \in  E$ with the common value of $C_D = 6$ (left) and also those for $C_D \in E$ with $C_U = 6$ (right).  The up-pointing and down-pointing triangles show the points at $a^*$ and $b^*$, respectively.}  \label{value_function_plot_quadratic}
\end{center}
\end{figure}

\subsection{Linear case}  
We shall then consider the linear case with $f \equiv f_L$ as in \eqref{def_f_L} with $\alpha^- = \alpha^+ = 1$.   Figure \ref{Gamma_plot_linear} shows the functions $b \mapsto \Gamma(a, b)$ and $b \mapsto \gamma(a, b)$ for $a = \underline{a}, (\underline{a} + a^*)/2, a^*, (a^* + \overline{a})/2, \overline{a}$ with the common values $C_U = C_D = 10$.  A noticeable difference with the quadratic case in Figure \ref{Gamma_plot_quadratic} is that the function $\gamma(a, \cdot)$ is not differentiable at $b = 0$. Moreover, $\Gamma(\underline{a}, \cdot)$ converges to a finite value.  Recall that the limit is positive if and only if $a^* = \underline{a}$ and $b^* = \infty$.   We can confirm that Assumption  \ref{assump_dec_inc} is indeed satisfied. 

Similarly to Figure \ref{value_function_plot_quadratic}, we show in Figure \ref{value_function_plot_linear} the value functions using the same parameter set for $C_U$ and $C_D$.  Here, we exclude the case $C_U = 35$ because it violates \eqref{linear_case_assump}.  Moreover, the case $C_D = 35$ is an example of \textbf{Case 2} because $C_D \geq \alpha^+ /q$ as in Proposition \ref{proposition_linear}.  We see that because the tail of the function $f$ grows/decreases slower than the quadratic case, the levels $(a^*, b^*)$ for this linear case are more sensitive to the values of $C_U$ and $C_D$. 


 \begin{figure}[htbp]
\begin{center}
\begin{minipage}{1.0\textwidth}
\centering
\begin{tabular}{cc}
 \includegraphics[scale=0.58]{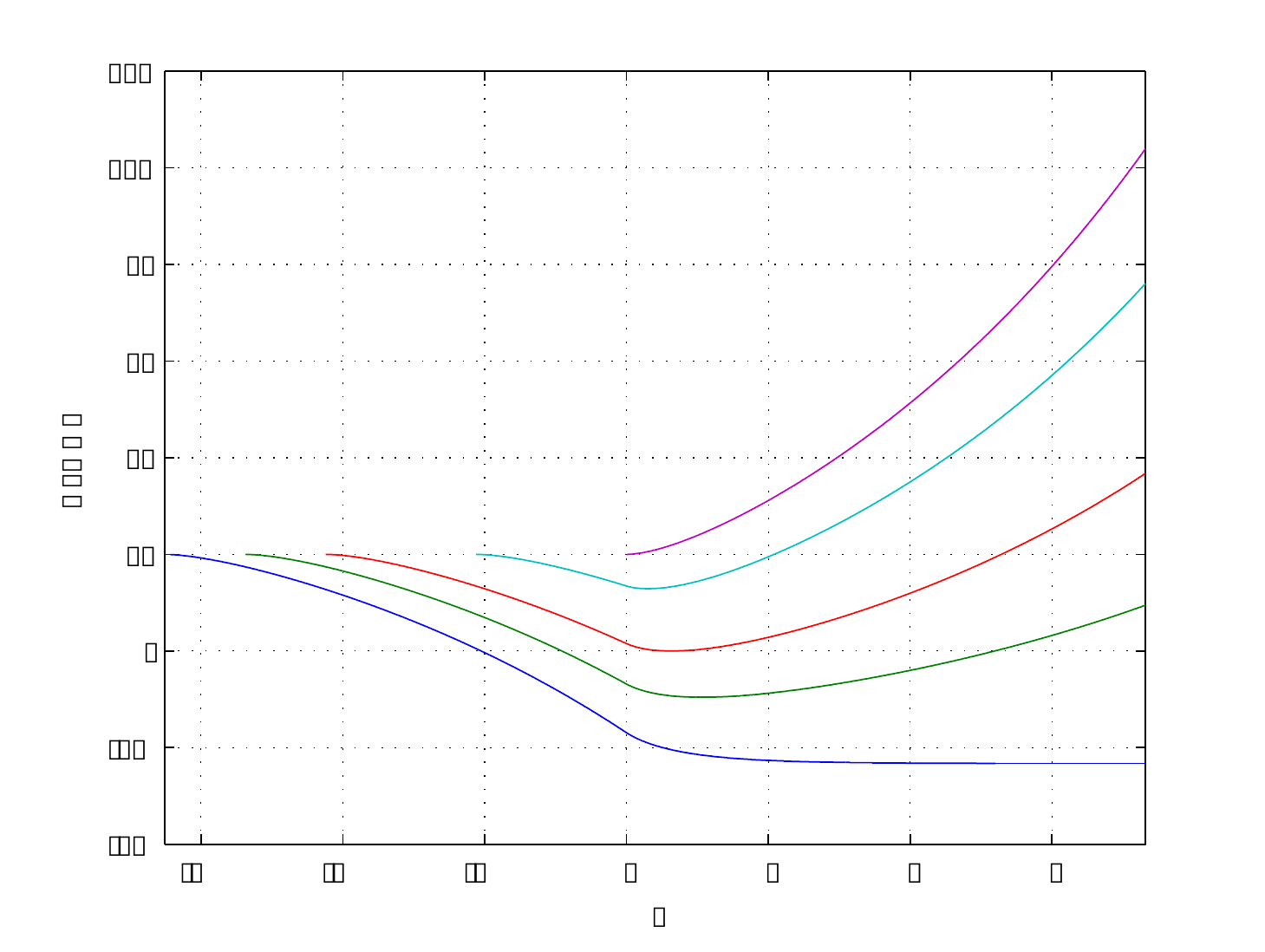} & \includegraphics[scale=0.58]{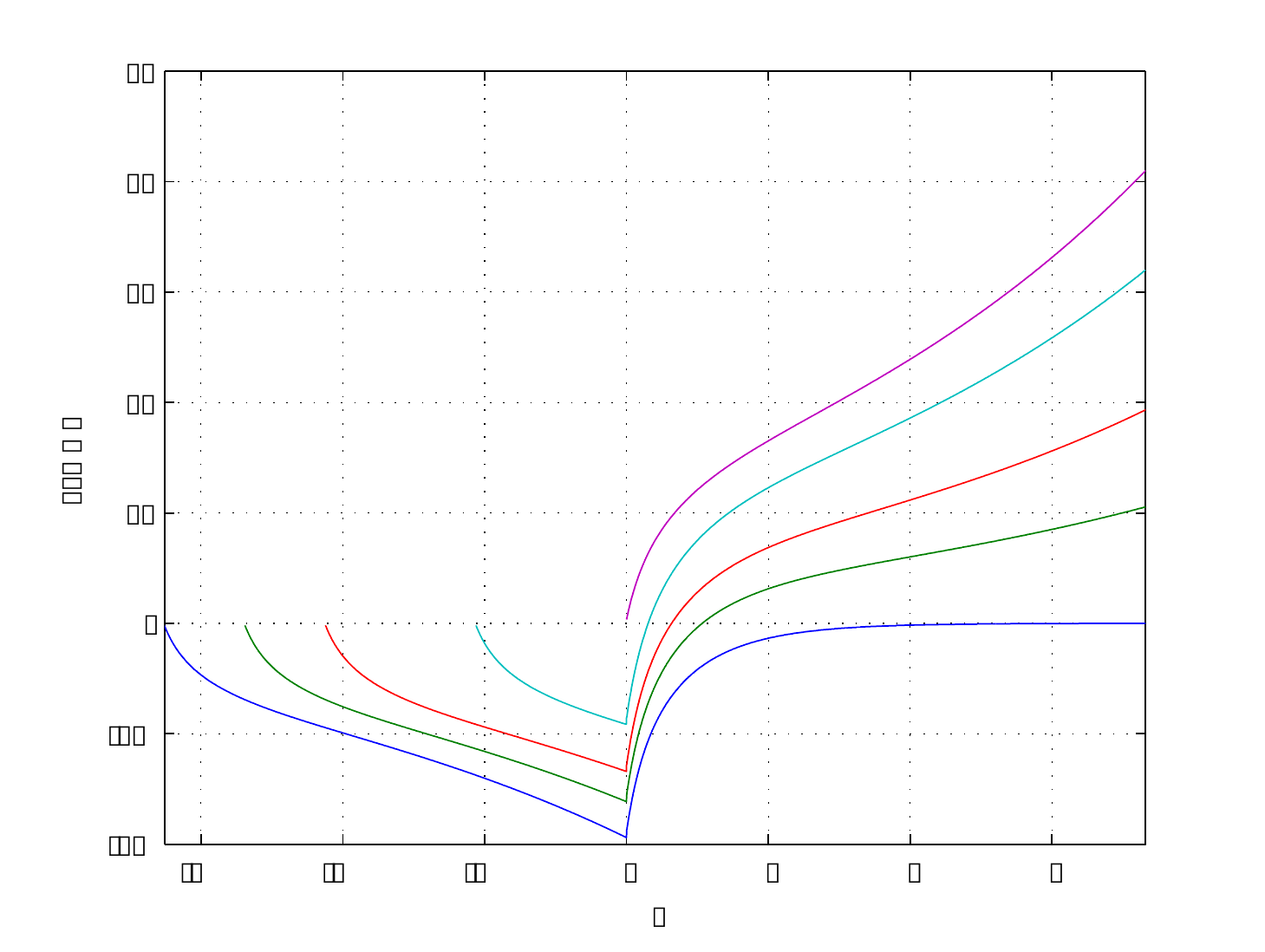}  \\
$\Gamma(a, \cdot)$ & $\gamma(a, \cdot)$  
\end{tabular}
\end{minipage}
\caption{Linear case: the plots of $b \mapsto \Gamma(a, b)$ (left) and $b \mapsto \gamma(a,b)$ (right) for $a = \underline{a}, (\underline{a} + a^*)/2, a^*, (a^* + \overline{a})/2, \overline{a}$.}  \label{Gamma_plot_linear}
\end{center}
\end{figure}

 \begin{figure}[htbp]
\begin{center}
\begin{minipage}{1.0\textwidth}
\centering
\begin{tabular}{cc}
 \includegraphics[scale=0.58]{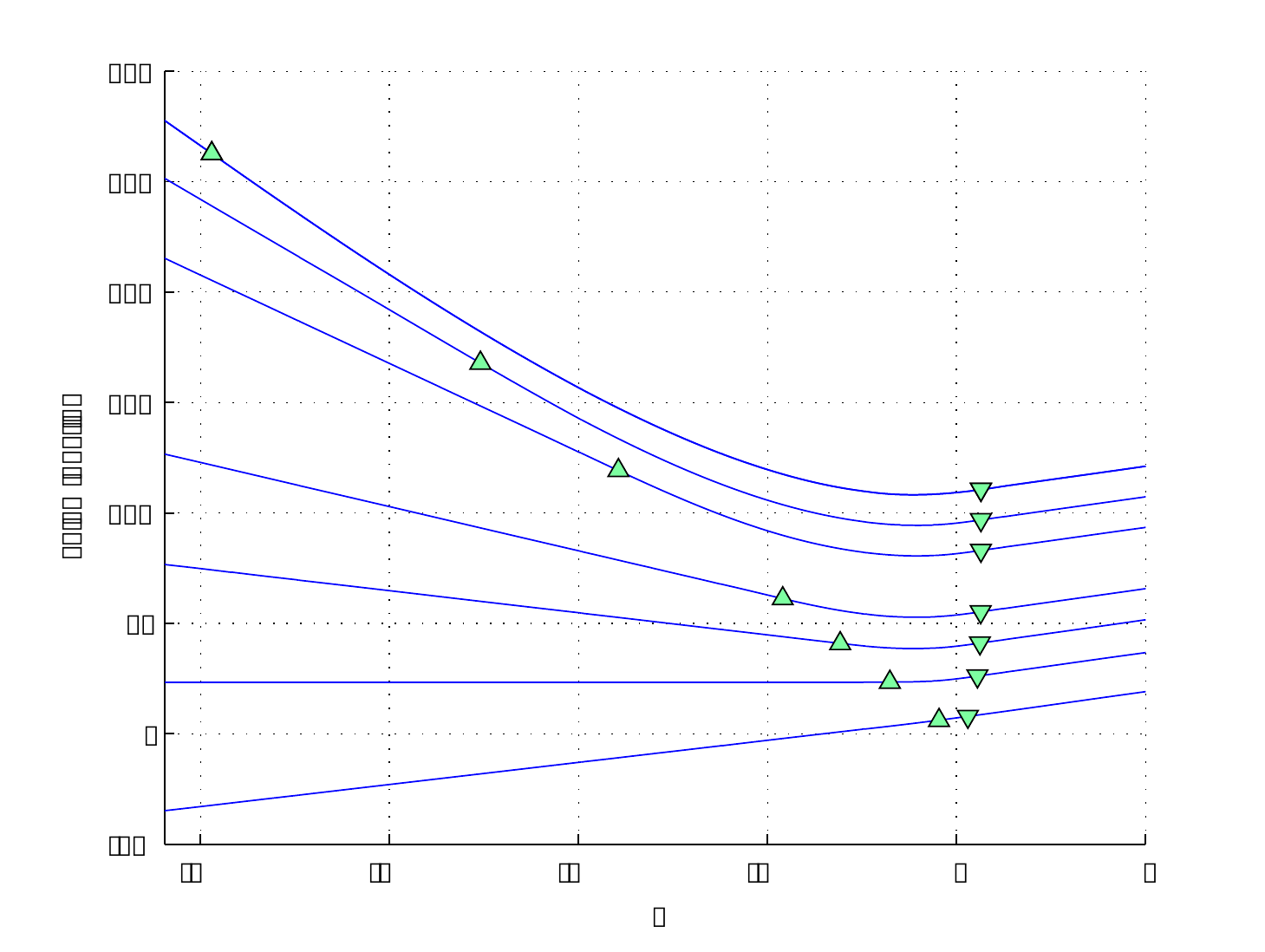} & \includegraphics[scale=0.58]{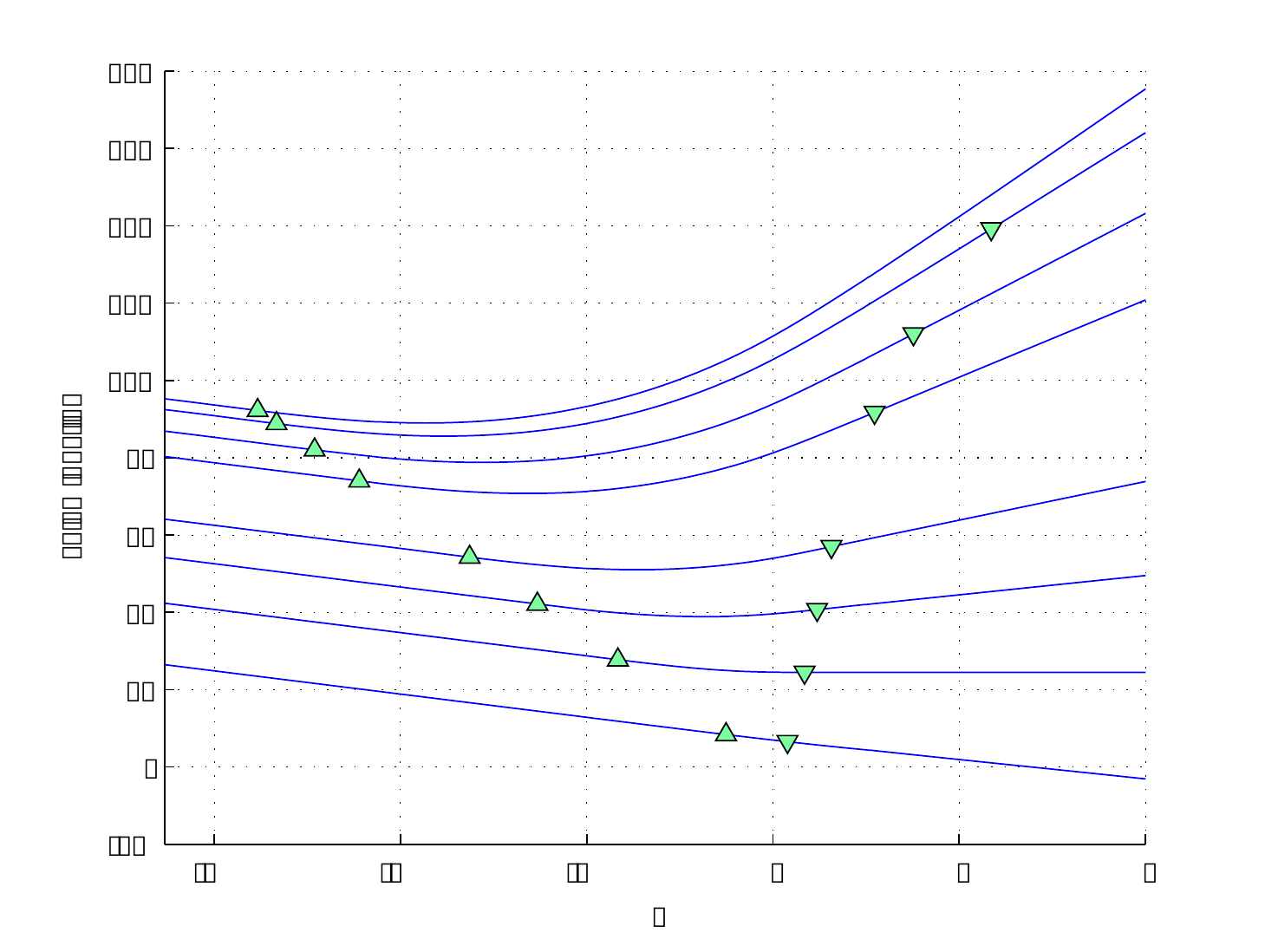}  \\
with respect to $C_U$ & with respect to $C_D$
\end{tabular}
\end{minipage}
\caption{Linear case: the plots of the value function for  $C_U \in  E \backslash \{35 \}$ with the common value of $C_D = 6$ (left) and also those for $C_D \in E$ with $C_U = 6$ (right).  The largest function in the right plot shows the case with $C_D = 35$ where the optimal barrier levels are given by $a^* = \underline{a}$ and $b^* = \infty$.}  \label{value_function_plot_linear}
\end{center}
\end{figure}

\section*{Acknowledgements}
We thank the anonymous referee for constructive comments and suggestions. E.\ J.\ Baurdoux was visiting CIMAT, Guanajuato when part of this work was carried out and he is grateful for their hospitality and support. K.\ Yamazaki is in part supported by MEXT KAKENHI grant numbers  22710143 and 26800092, JSPS KAKENHI grant number 23310103, the Inamori foundation research grant, and the Kansai University subsidy for supporting young scholars 2014.
\bibliographystyle{abbrv}
\bibliographystyle{apalike}

\bibliographystyle{agsm}
\bibliography{dual_model_bib}

\end{document}